\documentclass{amsart}

\usepackage[nobysame]{amsrefs}

\AtBeginDocument{\def\MR#1{}}

\usepackage{currfile}

\usepackage[english]{babel} 
\usepackage[UKenglish]{datetime}

\usepackage{ae, aecompl}
\usepackage{amscd}
\usepackage{amsfonts}
\usepackage{amsmath}
\usepackage{amssymb}
\usepackage{amsthm}
\usepackage{cases}
\usepackage{chngcntr}
\usepackage{cite}
\usepackage{color}
\usepackage{graphicx}
\usepackage{latexsym}
\usepackage{bbm} 
\usepackage{bbold}
\usepackage{mathtools}
\usepackage{microtype}
\usepackage{qsymbols}
\usepackage[Symbolsmallscale]{upgreek} 
\usepackage[active]{srcltx}
\usepackage{tikz}
\usepackage{url}
\usepackage{verbatim} 

\usepackage{enumitem}
\usepackage{cleveref}

\newcommand{\bbfont}{\mathbb}

\usepackage{ifthen}

\makeatletter
\newcommand{\tfs}[1]
{
\ifthenelse{\equal{\f@shape}{n}}{\ensuremath{\mathrm{#1}}}
	{\ifthenelse{\equal{\f@shape}{sc}}{\ensuremath{\mathrm{#1}}}
		{\ifthenelse{\equal{\f@shape}{it}}{\ensuremath{\mathit{#1}}}
			{\ifthenelse{\equal{\f@shape}{sl}}{\ensuremath{\mathit{#1}}}{}	
			}
		}
	}
}
\makeatother

\makeatletter
\newcommand{\btfs}[1]
{
\ifthenelse{\equal{\f@shape}{n}}{\ensuremath{\mathrm{#1}}}
	{\ifthenelse{\equal{\f@shape}{sc}}{\ensuremath{\mathrm{#1}}}
		{\ifthenelse{\equal{\f@shape}{it}}{\ensuremath{\mathit{#1}}}
			{\ifthenelse{\equal{\f@shape}{sl}}{\ensuremath{\mathit{#1}}}{}	
			}
		}
	}
}
\makeatother

\newcommand{\CC}{{\bbfont C}}

\newcommand{\NN}{{\bbfont N}}

\newcommand{\TT}{{\bbfont T}}

\newcommand{\ZZ}{{\bbfont Z}}

\newcommand{\upc}{{\mathrm{c}}}

\newcommand{\upB}{{\mathrm{B}}}

\newcommand{\ulp}{{\textup{(}}}
\newcommand{\urp}{{\textup{)}}}

\newcommand{\uppars}[1]{\ulp #1\urp}

\newcommand{\norm}[1]{{\lVert #1 \rVert}}

\newcommand{\braces}[1]{{\{ #1\}}}

\newcommand{\lrnorm}[1]{{\left\lVert #1 \right\rVert}}

\newcommand{\lrbrackets}[1]{{\left[ #1\right]}}
\newcommand{\lrpars}[1]{{\left( #1\right)}}

\newcommand{\Bigpars}[1]{{\Big( #1\Big)}}

\newcommand{\biggpars}[1]{{\bigg( #1\bigg)}}

\newcommand{\desset}[1]{\braces{\,#1\,}}

\newcommand{\bounded}{\upB}

\newcommand{\Cstar}{\ensuremath{\tfs{C}^\ast}}
\newcommand{\Calgebra}{\Cstar\!-algebra}
\newcommand{\Calgebras}{\Cstar\!-algebras}

\renewcommand{\ker}{\operatorname{ker}} 
\newcommand{\ran}{\operatorname{ran}}

\theoremstyle{plain}

\newtheorem{theorem}{Theorem}[section]
\newtheorem{proposition}[theorem]{Proposition}
\newtheorem{lemma}[theorem]{Lemma}
\newtheorem{corollary}[theorem]{Corollary}

\newtheorem*{theorem*}{Theorem}
\newtheorem*{proposition*}{Proposition}
\newtheorem*{lemma*}{Lemma}
\newtheorem*{corollary*}{Corollary}
\newtheorem*{conjecture*}{Conjecture}
\newtheorem*{context*}{Fixed context}

\theoremstyle{definition}

\newtheorem{definition}[theorem]{Definition}
\newtheorem{example}[theorem]{Example}
\newtheorem{remark}[theorem]{Remark}

\newtheorem*{definition*}{Definition}
\newtheorem*{example*}{Example}
\newtheorem*{remark*}{Remark}
\newtheorem*{assumption*}{Assumption}
\newtheorem*{hypothesis*}{Hypothesis}
\newtheorem*{question*}{Question}
\newtheorem*{problem*}{Problem}
\newtheorem*{task*}{Task}
\newtheorem*{addendum*}{Addendum}
\newtheorem*{idea*}{Idea}
\newtheorem*{suggestion*}{Suggestion}

\setlist[enumerate,1]{label=\textup{(\arabic*)},ref=\arabic*}
\setlist[enumerate,2]{label=\textup{(\alph*)},ref=\arabic{enumi}.\alph*}
\setlist[enumerate,3]{label=\textup{(\roman*)},ref=\arabic{enumi}.\alph{enumii}.\roman*}
\setlist[enumerate,4]{label=\textup{(\Alph*)},ref=\arabic{enumi}.\alph{enumii}.\roman{enumiii}.\Alph*}

\crefname{theorem}{Theorem}{Theorems}
\crefname{proposition}{Proposition}{Propositions}
\crefname{lemma}{Lemma}{Lemmas}
\crefname{corollary}{Corollary}{Corollaries}
\crefname{conjecture}{Conjecture}{Conjectures}
\crefname{definition}{Definition}{Definitions}
\crefname{example}{Example}{Examples}
\crefname{remark}{Remark}{Remarks}
\crefname{assumption}{Assumption}{Assumptions}
\crefname{hypothesis}{Hypothesis}{Hypotheses}
\crefname{question}{Question}{Questions}
\crefname{problem}{Problem}{Problems}
\crefname{addendum}{Addendum}{Addenda}
\crefname{idea}{Idea}{Ideas}
\crefname{suggestion}{Suggestion}{Suggestions}
\crefname{context}{Context}{Contexts}

\crefname{equation}{equation}{equations}
\crefname{enumi}{part}{parts}
\crefname{enumii}{part}{parts}
\crefname{enumiii}{part}{parts}
\crefname{enumiv}{part}{parts}

\newcommand{\enclosepart}[1]{(#1)}
\newcommand{\partref}[1]{\enclosepart{\ref{#1}}}

\numberwithin{equation}{section}

\allowdisplaybreaks 

\newcommand{\Hi}{H^{\mathrm{iso}}}
\newcommand{\Hu}{H^{\mathrm{uni}}}
\newcommand{\idop}{{\mathbf 1}}
\newcommand{\dnci}{doubly non-commuting isometries}
\newcommand{\Piso}{P^{\mathrm{iso}}}
\newcommand{\Pu}{P^{\mathrm{uni}}}
\renewcommand{\S}{V}
\newcommand{\standard}{{\mathcal V}}
\newcommand{\Ue}{{\mathcal U}}
\newcommand{\universal}{\TT}
\newcommand{\V}{V}
\newcommand{\Vtuple}{\ensuremath{(\V_1,\dotsc,\V_n)}}

\newcommand{\wdata}{{\mathcal{D}}}
\newcommand{\z}{z}
\newcommand{\zeop}{{\mathbf 0}}
\newcommand{\zz}{\overline{z}}

\begin{document}


\title[Doubly non-commuting isometries]{The structure of doubly non-commuting isometries}

\author[M.\ de Jeu]{Marcel de Jeu}

\address{ Mathematical Institute, Leiden University, P.O.\ Box 9512, 2300 RA Leiden,
	The Netherlands;
	and	Department of Mathematics and Applied Mathematics, University of Pretoria, Cor\-ner of Lynnwood Road and Roper Street, Hatfield 0083, Pretoria,
	South Africa}
\email{mdejeu@math.leidenuniv.nl}

\author[P.R.\ Pinto]{Paulo R.\ Pinto}
\address{Department of Mathematics, CAMGSD, Instituto Superior T\'{e}cnico, Universidade de Lisboa,
Av.\ Rovisco Pais 1, 1049-001 Lisbon, Portugal}
\email{ppinto@math.tecnico.ulisboa.pt}







\begin{abstract}Suppose that $n\geq 1$ and that, for all $i$ and $j$ with $1\leq i,j\leq n$ and $i\neq j$, $z_{ij}\in{\mathbb T}$ are given such that $z_{ji}=\overline{z}_{ij}$ for all $i\neq j$. If $V_1,\dotsc, V_n$ are isometries on a Hilbert space such that $V_i^\ast V_j^{\phantom{\ast}}\!=\overline{z}_{ij} V_j^{\phantom{\ast}}\!V_i^\ast$ for all $i\neq j$, then $(V_1,\dotsc,V_n)$ is called an $n$-tuple of doubly non-commuting isometries. The generators of non-commutative tori are well-known examples. In this paper, we establish a simultaneous Wold decomposition for $(V_1,\dotsc,V_n)$. This decomposition enables us to classify such $n$-tuples up to unitary equivalence. We show that the joint listing of a unitary equivalence class of a representation of each of the $2^n$ non-commutative tori that are naturally associated with the structure constants is a classifying invariant. A dilation theorem is also established, showing that an $n$-tuple of doubly non-commuting isometries can be extended to an $n$-tuple of doubly non-commuting unitary operators on an enveloping Hilbert space.
\end{abstract}


\maketitle


\section{Introduction and overview}\label{sec:introduction_and_overview}

Suppose that $n\geq 1$ and that, for all $i$ and $j$ with $1\leq i,j\leq n$ and $i\neq j$, $z_{ij}\in{\mathbb T}$ are given such that $z_{ji}=\zz_{ij}$ for all $i\neq j$. If $V_1,\dotsc, V_n$ are isometries on a Hilbert space such that $V_i^\ast V_j^{\phantom{\ast}}\!=\zz_{ij} V_j^{\phantom{\ast}}\!V_i^\ast$ for all $i\neq j$, then we shall refer to $(V_1,\dotsc,V_n)$ as an \emph{$n$-tuple of doubly non-commuting isometries}. In this paper, we shall show that, up to unitary equivalence, such $n$-tuples are uniquely determined by unitary equivalence classes of representations of the $2^n$ non-commutative tori that are naturally associated with the $\z_{ij}$. Equivalently, this gives a parameterization of the unitary equivalence classes of the representations of the universal \Calgebra\ that is generated by $n$ isometries satisfying the above relations.

The existing literature also suggests other names for our $n$-tuples. In \cite{weber:2013}, where $n=2$, the corresponding universal \Calgebra\ is called the tensor twist of the two isometries. In  \cite{kabluchko:2001} and \cite{proskurin_2000}, concerned with general $n$, no particular terminology is employed. In the case where $\z_{ij}=1$ for all $i$ and $j$, \cite{popescu:2016} and \cite{sarkar:2014} speak of doubly commuting isometries, and \cite{an_huef_raeburn_tolich:2015} of star-commuting (power partial) isometries. Since the non-commuting relation $V_i^\ast V_j^{\phantom{\ast}}\!=\zz_{ij} V_j^{\phantom{\ast}}\!V_i^\ast$ implies a second non-commuting relation $V_i V_j=\z_{ij} V_j V_i$ (this goes back to \cite{jorgensen_proskurin_samoilenko:2005}; see \cref{res:implied_relations} below), we believe that our terminology is justifiable. It also suggests a relation with the non-commutative tori that, in fact, exists and is an essential part of the picture.

We shall now give a combined overview and discussion of the paper.

Section~\ref{sec:space_decomposition_for_one_isometry} is concerned with the space decomposition that underlies the classical Wold decomposition of one isometry. This is briefly reviewed, and supplemented with some results that, although easy, are convenient tools in the sequel. In the case of one isometry, the identity operator is the sum of two projections, corresponding to the purely isometric and the unitary part of the operator in the Wold decomposition.  Furthermore, it is possible to write each of these projections as a strong operator limit in terms of the isometry and its adjoint; see \cref{eq:Piso_as_series,eq:Puni_as_limit}. These two facts will be the key to a relatively smooth proof of the space decomposition (and then of the subsequent Wold decomposition) for arbitrary $n$-tuples.

In Section~\ref{sec:space_decomposition_general_case} the case of a general $n$-tuple \Vtuple\ is taken up. Taking the product of the decompositions of the identity operator for the various $\V_i$, one obtains a decomposition of the space into $2^n$ simultaneously reducing subspaces, with the property that in each of these every $\V_i$ acts as a pure isometry or as a unitary operator; see \cref{res:space_decomposition_types_of_actions}. Each of the $2^n$ corresponding projections is a product of $n$ projections taken from the decompositions of the identity operator for the various $\V_i$. Such a product of projections is then further analysed by invoking the appropriate strong operator limits from \cref{eq:Piso_as_series,eq:Puni_as_limit} for its factors. After identifying the various range projections of partial isometries in the result, a structure theorem for each of the $2^n$ space components is then obtained in terms of a wandering subspace; see  \cref{res:space_decomposition_for_given_types_of_actions}. In the case where all $\z_{ij}$ are equal to 1 this can already be found as \cite[Theorem~3.1]{sarkar:2014}; we also refer to \cite{sarkar:2014} for an overview of the preceding literature on the Wold decomposition for $n$-tuples of (then) doubly commuting isometries. Our analysis for general structure constants continues from here, however, and the starting point for this continuation is to observe that the $\V_i$ that act as unitary operators on the space component at hand leave its wandering subspace invariant; see \cref{res:space_decomposition_for_given_types_of_actions} again. The ensuing actions of $2^n$ different non-commutative tori on their corresponding wandering subspaces will turn out to be the core of the simultaneous action of our $n$-tuple.

We would like to mention explicitly that the method of taking a product of various decompositions of the identity operator differs from inductive approaches as in \cite{an_huef_raeburn_tolich:2015,kabluchko:2001,proskurin_2000,sarkar:2014}. Employing such a product may be a more transparent way of working, although this remains a matter of taste. At any rate, it has the advantage that it could conceivably also be of use in other contexts, e.g.\ when all operators in an $n$-tuple are of a different type and induction may not be so easy to apply.

In Section~\ref{sec:Wold_decomposition_and_examples} we use the results from Section~\ref{sec:space_decomposition_general_case} to show that, up to unitary equivalence, all $n$-tuples \Vtuple\ of \dnci\ are a direct sum of $2^n$ so-called standard $n$-tuples. This Wold decomposition for all operators in the $n$-tuple simultaneously is the statement of \cref{res:wold_decomposition}; if all $\z_{ij}$ are equal to 1, this is a particular case of \cite[Theorem~2.25]{an_huef_raeburn_tolich:2015}. The $2^n$ standard $n$-tuples correspond to the $2^n$ components in the decomposition of the space from Section~\ref{sec:space_decomposition_general_case} as mentioned above. The structure of such a standard $n$-tuple is  completely explicit once the action of the pertinent non-commutative torus on the pertinent wandering subspace is given; see the material preceding \cref{res:unitary_equivalence_to_standard_realisation}. The actions of the non-commutative tori on the wandering spaces (described by the wandering data as defined in \cref{def:wandering_subspace_and_core}) should be thought of as the parameters for the $n$-tuple.

It is only in this Section~\ref{sec:Wold_decomposition_and_examples} that a natural class of examples of $n$-tuples of \dnci\ first appears. The structure results from Section~\ref{sec:space_decomposition_general_case}, that could conceivably be applicable only to operators on the zero space, inform us what such examples should look like. It is then easy to check that the ensuing Ansatz actually works, and this results in the standard $n$-tuples. In the irreducible case, the structure of these examples is already visible in \cite[Theorem~2]{proskurin_2000}. The proof of \cite[Theorem~2]{proskurin_2000} is only indicated; the absence of the framework of the general Wold decomposition as an aid in formulating such a proof can perhaps explain this. We shall include a strengthened version of \cite[Theorem~2]{proskurin_2000} in Section~\ref{sec:classification}; see \cref{res:irreducibility}. It gives a  parameterization of the unitary equivalence classes of the irreducible representations of the universal \Calgebra\ generated by $n$ isometries satisfying \cref{eq:relations_for_operators}, and it follows rather easily from the results in the present paper on general representations.

Section~\ref{sec:classification} is concerned with the unitary equivalence classes of $n$-tuples of \dnci\ or, equivalently, with the unitary equivalence classes of representations of the universal \Calgebra\ generated by isometries satisfying our relations. The result, formulated in \cref{res:classification}, has a certain aesthetic appeal: these classes are parameterized by the lists of $2^n$ unitary equivalence classes of representations of the $2^n$ non-commutative tori that are naturally associated with the given structure constants $z_{ij}$, containing one such class for each non-commutative torus.
In a worked example for the case $n=1$ it is then seen that the unitary equivalence class of an isometry is determined by the combination of an equivalence class of a representation of the non-commutative $0$-torus and an equivalence class of a representation of the non-commutative $1$-torus. The classifying invariants for an isometry (the multiplicity of the unilateral shift and the equivalence class of its unitary component) are thus retrieved from a more general framework.

We include a dilation theorem in Section~\ref{sec:dilation}; see \cref{res:dilation}. As for $n=1$, now that a Wold decomposition is available, this is merely a matter of extending the range of indices from the non-negative to all integers where needed. 

\begin{remark*}
There are certain standard $n$-tuples that are particularly elementary. As it turns out, these give faithful representations of the universal \Calgebras\ that are generated by $n$ doubly non-commuting isometries where specified generators are required to be even unitary. The known  faithfulness of the Fock representation of one of these algebras (see \cite[Proposition~8]{proskurin_2000} and \cite[Corollary~1]{kabluchko:2001}) is then a special case. We refer to \cref{rem:part_II} for some more comments. We shall report on these universal \Calgebras, their interrelations, and their representations in a separate paper, for which the current paper also serves as a preparation.
\end{remark*}

We conclude by listing our conventions. First of all, we shall always work in the following context.

\begin{context*} $H$ is a Hilbert space, $n\geq 1$, and \Vtuple\ is an $n$-tuple of isometries on $H$. For all $i$ and $j$ with $1\leq i,j\leq n$ and $i\neq j$, $z_{ij}\in\TT$ are given such that $z_{ji}=\zz_{ij}$ for all such $i$ and $j$, and the isometries $\V_1,\dotsc, \V_n$ satisfy
\begin{equation}\label{eq:relations_for_operators}
\V_i^\ast\V_j=\zz_{ij} \V_j\V_i^\ast
\end{equation}
for all such $i$ and $j$.
\end{context*}

We shall then say that \Vtuple\ is an $n$-tuple of \dnci, without any further reference to the structure constants in terminology or notation. If $n=1$, then the $n$-tuple reduces to a given single isometry without further requirements. If the need arises, we shall sometimes write $z_{i,j}$ instead of $z_{ij}$.
With the sole exception of \cref{res:equivalent_tuple_must_have_same_constants}, we shall not vary the structure constants $\z_{ij}$.

All Hilbert spaces are complex, and subspaces are always closed subspaces. The bounded operators on $H$ are denoted by $\bounded(H)$, and we write $\zeop$ and $\idop$ for the zero and the identity operator on $H$, respectively. Projections are always orthogonal projections. An empty product of operators on $H$ is to be read as $\idop$. If $T\in\bounded(H)$ and $L$ is a subspace of $H$ that is invariant under $T$, then $T|_L$ is the restriction of $T$ to $L$.

If $H$ and $H^\prime$ are Hilbert spaces, and $(T_1,\dotsc,T_n)$ and $(T_1^\prime,\dotsc,T_n^\prime)$ are $n$-tuples of operators on $H$ and $H^\prime$, respectively, then we say that $(T_1,\dotsc,T_n)$ and $(T_1^\prime,\dotsc,T_n^\prime)$ are unitarily equivalent if there exists an isometry between $H$ and $H^\prime$ that is a unitary equivalence for all pairs $T_i$ and $T_i^\prime$ with $1\leq i\leq n$ simultaneously.

If $A\subseteq\{1,\dotsc,n\}$ is a (possibly empty) set of indices, then we shall write $|A|$ for its number of elements, and let $A^\upc$ denote the complement of $A$ in $\{1,\dotsc,n\}$.

Finally, we let $\NN=\{1,2,\dots\}$ and $\NN_0=\{0,1,2,\dots\}$.


\section{Space decomposition for one isometry}\label{sec:space_decomposition_for_one_isometry}


The Wold decomposition (see e.g.
\cite[Theorem~3.5.17]{murphy_C-STAR-ALGEBRAS_AND_OPERATOR_THEORY:1990}) for an isometry on a Hilbert space asserts that it is the direct sum of a number of copies of the unilateral shift and of a unitary operator, where each summand can be zero. The first step in proving this is to decompose the space as a Hilbert direct sum of a subspace on which the operator acts as a pure isometry on the one hand, and a subspace on which it acts as a unitary operator on the other hand. In the second step, which is almost just an afterthought, the aforementioned structure of the operator is then clear from the available explicit decomposition of the summand where the operator acts as a pure isometry.

For general $n$-tuples of \dnci\ the global approach is the same. The first step is to decompose the space (see \cref{res:space_decomposition_types_of_actions,res:space_decomposition_for_given_types_of_actions}), and the second one is to use this decomposition as a starting point for a description of the structure of the $n$-tuple (see \cref{res:wold_decomposition}).

This section is a preparation for the first step for the general case. We give a short proof for the space decomposition for the case $n=1$ (see \cref{res:space_decomposition_for_one_isometry}), and add a few small results that, in later sections, will be very convenient to have been mentioned explicitly.

\emph{Throughout this section, $\S$ is an isometry on a Hilbert space $H$}.

We start with the decomposition of the space. Since $\S^\ast \S=\idop$, the subspaces $\S^k(\ker \S^\ast)$ and $\S^{k^\prime}(\ker \S^\ast)$ are easily seen to be pairwise orthogonal if $k,k^\prime\geq 0$ and $k\neq k^\prime$. Using an anticipating notation, we can, therefore, define

\begin{align*}
\Hi&\coloneqq\bigoplus_{k=0}^\infty \S^k(\ker \S^\ast)\\
\intertext{as  a Hilbert direct sum. Furthermore, we let}
\Hu&\coloneqq\bigcap_{k=0}^\infty \S^k(H).
\end{align*}

We denote by $\Piso$ the projection onto $\Hi$ and by $\Pu$ the projection onto $\Hu$.

The following fact is classical; see \cite[Theorem~I.1.1]{sz_nagy_foias_bercovici_kerchy_HARMONIC_ANALYSIS_OF_OPERATORS_ON_HILBERT_SPACE_SECOND_EDITION:2010} for a proof, for example.


\begin{proposition}\label{res:space_decomposition_for_one_isometry}
$\Hi$ and $\Hu$ are both $\S$-reducing subspaces of $H$, and we have $H=\Hi\bigoplus\Hu$ as a Hilbert direct sum.
\end{proposition}


\begin{remark}
As $\S$ is unitary on $\Hu$ (see \cite[Theorem~I.1.1]{sz_nagy_foias_bercovici_kerchy_HARMONIC_ANALYSIS_OF_OPERATORS_ON_HILBERT_SPACE_SECOND_EDITION:2010}; it is also a consequence of \cref{res:unitary_reducing_subspaces}, below), the Wold decomposition of $\S$ is immediate from \cref{res:space_decomposition_for_one_isometry}: the copies of the unilateral shift correspond to the elements of an orthonormal basis of $\ker \S^\ast$.
\end{remark}

As a first preparation for Section~\ref{sec:space_decomposition_general_case}, we note that, trivially,
\begin{equation}\label{eq:identity_decomposition}
\idop=\Piso+\Pu.
\end{equation}
As a second preparation, we shall now express each of the summands in terms of range projections of partial isometries.  For this, we recall that the projection onto the range of a partial isometry $T$ is given by $T T^\ast$; see \cite[p.~23]{davidson_C-STAR-ALGEBRAS_BY_EXAMPLE:1996}.

Since $\ker \S^\ast=(\ran \S)^\bot$, the projection onto $\ker \S^\ast$ is $\idop-\S \S^\ast$.
Therefore, for $k\geq 0$, the range of $\S^k(\idop-\S \S^\ast)$ is $\S^k (\ker \S^\ast)$.
Since $\S^k$ is an isometry and $(\idop-\S \S^\ast)$ is a projection, $\S^k(\idop-\S \S^\ast)$ is a partial isometry, and its range projection is then $\S^k(\idop-\S \S^\ast)(\idop-\S \S^\ast)^\ast \S^{\ast k}=\S^k(\idop-\S \S^\ast) \S^{\ast k}$.
All in all, we see that $\S^k(\idop-\S \S^\ast)\S^{\ast k}$ is the projection onto $\S^k (\ker \S^\ast)$. We know from their interpretations (this can also easily be verified algebraically) that the projections for different $k$ are orthogonal. Consequently,

\begin{equation}\label{eq:Piso_as_series}
\Piso=\sum_{k=0}^\infty \S^k(\idop-\S \S^{\ast})\S^{\ast k},
\end{equation}
where the series converges to $\Piso$ in the strong operator topology as a consequence of \cite[Lemma~I.6.4]{davidson_C-STAR-ALGEBRAS_BY_EXAMPLE:1996}.

The projection onto $\bigcap_{k=0}^\infty \S^k (H)$ is the infimum of the decreasing sequence of projections onto the spaces $\S^k(H)$, i.e.\ the infimum of the projections $\S^k \S^{\ast k}$. Again by \cite[Lemma~I.6.4]{davidson_C-STAR-ALGEBRAS_BY_EXAMPLE:1996}, we see that
\begin{equation}\label{eq:Puni_as_limit}
\Pu=\mathrm{SOT-}\lim_{k\to \infty} \S^k \S^{\ast k},
\end{equation}

The equations~\eqref{eq:identity_decomposition}, \eqref{eq:Piso_as_series}, and~\eqref{eq:Puni_as_limit} are at the heart of the space decomposition for the general case in Section~\ref{sec:space_decomposition_general_case}.

As announced in the introduction of this section, we shall now collect a few results on invariant and reducing subspaces such that the restricted operator is unitary or purely isometric.
\cref{res:unitary_reducing_subspaces,res:pure_isometry_reducing_subspaces} will be unified and generalised as a part of \cref{res:space_decomposition_types_of_actions}, and \cref{res:pure_isometry_and_unitary_splitting} will be generalised as \cref{res:general_pure_isometry_and_unitary_splitting}.

We start with the unitary case.

\begin{lemma}\label{res:unitary_subspace_is_in_Hu}\quad
\begin{enumerate}
\item Let $L$ be a $\S$-invariant subspace of $H$.
If $\S|_L$ is unitary, then $L$ reduces $\S$, and $L\subseteq \Hu$.
\item Let $L$ be a $\S$-reducing subspace of $H$. If $L\subseteq\Hu$, then $\S|_L$ is unitary.
\end{enumerate}
\end{lemma}
\begin{proof}
(1) It is easy to check that $L$ reduces $\S$.
The hypothesis implies that $\S^k(L)=L$ for all $k\geq 0$.
Hence
$L=\bigcap_{k=0}^\infty \S^k(L)\ \subseteq \ \bigcap_{k=0}^\infty \S^k(H)=\Hu.$

(2) We need to show that $\ran(\S|_L)=L$, or equivalently, that $\left(\ran(\S|_L)\right)^{\bot_L}=\{0\}$, where the orthogonal complement is taken in $L$.
But
\[
\left(\ran(\S|_L)\right)^{\bot_L}=\ker(\S|_L)^\ast=
\ker(\S^\ast|_L).
\]
Since $L\subseteq \Hu\subseteq \S(H)$, and $\S^\ast$ is injective on $\S(H)$ because $\S^\ast \S=\idop$, we see that $\ker(\S^\ast|_L)=\{0\}$, as required.
\end{proof}




The following is now obvious.

\begin{proposition}\label{res:unitary_reducing_subspaces}
Let $L$ be a $\S$-reducing subspace of $H$.
Then the following are equivalent:
\begin{enumerate}
	\item $\S$ is unitary on $L$;
	\item $L\subseteq \Hu$;
	\item $P_L\Pu=P_L$, where $P_L$ denotes the projection onto $L$.
\end{enumerate}
\end{proposition}


The following notion of a pure isometry (the absence of all non-trivial unitarity) is somewhat more intuitive than what is usually found in the literature, which is that $\Hu$ should be the zero subspace; see e.g.\ \cite[p.~154]{douglas_BANACH_ALGEBRA_TECHNIQUES_IN_OPERATOR_THEORY_SECOND_EDITION:1998} or \cite[p.~113]{arveson_A_SHORT_COURSE_ON_SPECTRAL_THEORY:2002}. \cref{res:pure_isometry} shows that the two definitions are, in fact, equivalent.

\begin{definition}\label{def:pure_isometry}
Let $L$ be a $\S$-invariant subspace of $H$. Then $\S$ is a \emph{pure isometry on $L$}, or \emph{$\S|_L$ is a pure isometry}, if $\{0\}$ is the only $\S$-invariant subspace of $L$ on which $\S$ is unitary.
\end{definition}

The following is a consequence of \cref{res:space_decomposition_for_one_isometry}, the first part of \cref{res:unitary_subspace_is_in_Hu}, and \cref{res:unitary_reducing_subspaces}.
\begin{proposition}\label{res:pure_isometry}
The following are equivalent:
\begin{enumerate}
	\item\label{part:pure_isometry} $\S$ is a pure isometry on $H$;
	\item\label{part:isometric_subspace_entire_space} $\Hi=H$;
	\item\label{part:unitary_subspace_is_zero} $\Hu=\{0\}$;
	\item\label{part:reducing_unitary_subspace_is_zero} $\{0\}$ is the only $\S$-reducing subspace of $H$ on which $\S$ is unitary.
\end{enumerate}
\end{proposition}



\begin{lemma}\label{res:subspace_of_Hi_is_pure_isometry}\quad
\begin{enumerate}
\item Let $L$ be a $\S$-invariant subspace of $H$. If $L\subseteq \Hi$, then $\S$ is a pure isometry on $L$.

\item Let $L$ be a $\S$-reducing subspace of $H$. If $\S$ is a pure isometry on $L$, then $L\subseteq \Hi$.
\end{enumerate}
\end{lemma}

\begin{proof}
(1) In view of \cref{res:pure_isometry}, we need to show that $L^{{\mathrm{uni}}}=\bigcap_{k=0}^\infty (\S|_L)^k(L)=\{0\}$.
For this it is sufficient to show that $\bigcap_{k=0}^\infty \S^k(\Hi)=\{0\}$. This follows from the observation that $\S^k(\Hi)=\bigoplus_{n=k}^\infty \S^n(\ker \S^\ast)$ for every $k$.

(2) We know from \cref{res:space_decomposition_for_one_isometry} and \cref{res:pure_isometry} that
\[
L=\bigoplus_{k=0}^\infty (\S|_L)^k\left(\ker (\S|_L)^\ast\right)\oplus \bigcap_{k=0}^\infty (\S|_L)^k(L) =\bigoplus_{k=0}^\infty (\S|_L)^k\left(\ker (\S|_L)^\ast\right).
\]
Since $L$ is a $\S$-reducing subspace of $H$, we have $(\S|_L)^\ast=\S^\ast|_L$.
Hence
$\ker(\S|_L)^\ast\subseteq \ker \S^\ast$, and then
$L\subseteq \bigoplus_{k=0}^\infty
{(\S|_L)}^k\left(\ker \S^\ast\right)
\subseteq \Hi.$
\end{proof}

Since $\Hi$ is invariant under $\S$ by \cref{res:space_decomposition_for_one_isometry},
$\S$ is a pure isometry on $\Hi$.




The following is now clear.

\begin{proposition}\label{res:pure_isometry_reducing_subspaces}
Let $L$ be a $\S$-reducing subspace of $H$. Then the following are equivalent:
\begin{enumerate}
	\item $\S$ is a pure isometry on $L$;
	\item $L\subseteq \Hi$;
	\item $P_L\Piso=P_L$.
\end{enumerate}
\end{proposition}

The unitary counterpart of \cref{res:pure_isometry}, proved using
\cref{res:space_decomposition_for_one_isometry,res:unitary_reducing_subspaces}, and \cref{res:subspace_of_Hi_is_pure_isometry}, is as follows.

\begin{proposition}\label{res:unitary}
The following are equivalent:
\begin{enumerate}
	\item\label{part:unitary} $\S$ is unitary on $H$;
	\item\label{part:unitary_subspace_entire_space} $\Hu=H$;
	\item\label{part:isometric_subspace_is_zero} $\Hi=\{0\}$;
	\item\label{part:reducing_purely_isometric_subspace_is_zero} $\{0\}$ is the only $\S$-reducing subspace of $H$ on which $\S$ is a pure isometry.
\end{enumerate}
\end{proposition}


The next result follows from \cref{res:space_decomposition_for_one_isometry,res:unitary_reducing_subspaces,res:pure_isometry_reducing_subspaces}.

\begin{proposition}\label{res:pure_isometry_and_unitary_splitting}\quad
\begin{enumerate}
	\item If $L^{{\mathrm{iso}}}$ and $L^{{\mathrm{uni}}}$ are $\S$-reducing subspaces such that $\S|_{L^{{\mathrm{iso}}}}$ and $\S|_{L^{{\mathrm{uni}}}}$ are a pure isometry and unitary, respectively, then $L^{{\mathrm{iso}}}\ \bot \ L^{{\mathrm{uni}}}$.
	\item If $H=L^{{\mathrm{iso}}}\oplus L^{{\mathrm{uni}}}$ \uppars{algebraically}, where $L^{{\mathrm{iso}}}$ and $L^{{\mathrm{uni}}}$ are $\S$-reducing subspaces such that $\S|_{L^{{\mathrm{iso}}}}$ and $\S|_{L^{{\mathrm{uni}}}}$ are a pure isometry and unitary, respectively, then $L^{{\mathrm{iso}}}=\Hi$ and $L^{{\mathrm{uni}}}=\Hu$.
\end{enumerate}
\end{proposition}

%
%

%


\section{Space decomposition in the general case}\label{sec:space_decomposition_general_case}


We shall now establish a space decomposition for an $n$-tuple \Vtuple\ of \dnci. This is done in two parts. In the first part, the space is written as a Hilbert sum of (possibly zero) subspaces on which each of the $\V_i$ acts as a pure isometry or a unitary operator; see \cref{res:space_decomposition_types_of_actions}. This is an elementary consequence of the results in Section~\ref{sec:space_decomposition_for_one_isometry}. In the second part, which is more involved, each of the summands from the first step is written as a Hilbert direct sum of copies of a wandering subspace; see~\cref{res:space_decomposition_for_given_types_of_actions}. The method to obtain this is not an inductive procedure as in \cite{an_huef_raeburn_tolich:2015,kabluchko:2001,proskurin_2000,sarkar:2014}, but consists of multiplying $n$ decompositions of the identity operator and interpreting the result.

As a side remark let us note that, at this stage, it is not clear within the framework of the current paper that, for general $n\geq 2$, there are any non-zero examples of $n$-tuples of \dnci\ at all. We shall see in Section~\ref{sec:Wold_decomposition_and_examples}, however, that non-zero examples of a very simple nature exist. The results in the present section will guide us towards these examples.

We start by collecting a few algebraic results.

The first part of the next result and its proof can already be found in \cite[p.~2671]{jorgensen_proskurin_samoilenko:2005}.
It shows that the use of complex conjugation in \cref{eq:relations_for_operators} is not so unnatural after all.

\begin{lemma}\label{res:implied_relations}
For all $i\not=j$,
\begin{enumerate}
	\item\label{part:implied_relations_1} $\V_i\V_j=\z_{ij} \V_j\V_i$;
	\item\label{part:implied_relations_2} $\V_j^\ast \V_i^{\phantom{\ast}}\!=\z_{ij} \V_i^{\phantom{\ast}}\!\! \V_j^\ast$;
	\item\label{part:implied_relations_3} $\V_j^\ast \V_i^\ast=\zz_{ij} \V_i^\ast \V_j^\ast$.
\end{enumerate}
\end{lemma}

\begin{proof}
An easy computation shows that $(\V_i\V_j-\z_{ij} \V_j\V_i)^\ast(\V_i\V_j-\z_{ij} \V_j\V_i)=0$, which gives \partref{part:implied_relations_1}. The other parts follows by taking adjoints.
\end{proof}

As remarked in \cite[p.~2671]{jorgensen_proskurin_samoilenko:2005}, the relation $\V_1\V_2=\z \V_2\V_1$ for isometries $V_1,V_2$ and $\z \in\TT$ does not imply that $\V_1^\ast \V_2^{\phantom{\ast}}\!=\zz\V_2^{\phantom{\ast}}\!\V_1^\ast$. Although this implications holds true  for unitary operators, it is not valid in general. There is an elementary counterexample in \cite[Lemma 1.2]{weber:2013}.

\begin{corollary}\label{res:commuting_operators}
Let $k\geq 0$.
Then, for all $i\not=j$,
\begin{enumerate}
	\item\label{part:commuting_operators_1} $\V_i^{\phantom k}$ and $\V_j^k\V_j^{\ast k}$ commute;
	\item\label{part:commuting_operators_2} $\V_i^{\ast}$ and $\V_j^k\V_j^{\ast k}$ commute;
	\item\label{part:commuting_operators_3} $\V_i^{\phantom k}$ and $\V_i^{\ast}$ commute with $\V_j^k(\idop-\V_j^{\phantom\ast}\!\!\V_j^{\ast})\V_j^{\ast k}$.
\end{enumerate}
\end{corollary}

\begin{proof}
For part~~\partref{part:commuting_operators_1} we may suppose that $k,l\geq 1$.
Repeated use of \cref{res:implied_relations} and \cref{eq:relations_for_operators}, combined with $\z_{ji}=\zz_{ij}$, shows that
\[
\V_i^{\phantom{\ast}}\!\, \V_j^k\V_j^{\ast k}=\z_{ij}^{k} \V_j^k \V_i^{\phantom{\ast}}\!\!\V_j^{\ast k}={z}_{ij}^{k}\z_{ji}^{k} \V_j^k\V_j^{\ast k} \V_i^{\phantom{\ast}}\!\!=\V_j^k\V_j^{\ast k}\, \V_i^{\phantom{\ast}}\!,
\]
as claimed. Part~\partref{part:commuting_operators_2} follows from part~\partref{part:commuting_operators_1} by taking adjoints. Part~\partref{part:commuting_operators_3} is immediate from the parts~\partref{part:commuting_operators_1} and~\partref{part:commuting_operators_2}.
\end{proof}


\begin{lemma}\label{res:moving_factors}
Let $A=\{i_1,\dotsc,i_l\}\subseteq \{1,\dotsc,n\}$ be a \uppars{possibly empty} set of $l$ different indices, and let $k_{i_1},\dotsc,k_{i_l}\geq 0$ be exponents. Then
\begin{equation}\label{eq:moving_factors_1}
	\begin{split}
	&\lrbrackets{\V_{i_1}^{k_{i_1}}(\idop-\V_{i_1}^{\phantom{\ast}}\!\V_{i_1}^{\ast})\V_{i_1}^{\ast k_{i_1}}} \dotsm \lrbrackets{\V_{{i_l}}^{k_{i_l}}(\idop-\V_{i_l}^{\phantom{\ast}}\!\V_{i_l}^{\ast})\V_{i_l}^{\ast k_{i_l}}}\\
&\quad\quad =\lrbrackets{\V_{i_1}^{k_{i_1}}\dotsm \V_{i_l}^{k_{i_l}}}\Big[ (\idop-\V_{i_1}^{\phantom{\ast}}\!\V_{i_1}^{\ast}) \dotsm (\idop-\V_{i_l}^{\phantom{\ast}}\!\V_{i_l}^{\ast})\Big]\lrbrackets{\V_{i_l}^{\ast k_{i_l}}\dotsm \V_{i_1}^{\ast k_{i_1}}}
\end{split}
	\end{equation}
and
	\begin{equation}\label{eq:moving_factors_2}
	\lrbrackets{\V_{i_1}^{k_{i_1}}\V_{i_1}^{\ast k_{i_1}}}\dotsm\lrbrackets{\V_{i_l}^{k_{i_l}}\V_{i_l}^{\ast k_{i_l}}}=\lrbrackets{\V_{i_1}^{k_{i_1}}\dotsm\V_{i_l}^{k_{i_l}}}\lrbrackets{\V_{i_l}^{\ast k_{i_l}}\dotsm\V_{i_1}^{\ast k_{i_1}}}.
	\end{equation}
\end{lemma}

\begin{proof}
We prove~\cref{eq:moving_factors_1} by induction on $l$, the number of factors in the left hand side. For $l=0$ and $l=1$ all is clear. Assuming the statement for $l$, the induction hypothesis for the product of the first $l$ factors shows that the product with $(l+1)$ factors equals
\[
\V_{i_1}^{k_{i_1}}\dotsm\V_{i_l}^{k_{i_l}} (\idop-\V_{i_1}^{\phantom{\ast}}\!\V_{i_1}^{\ast})\dotsm(\idop-\V_{i_l}^{\phantom{\ast}}\!\V_{i_l}^{\ast}) \V_{i_l}^{\ast k_{i_l}}\dotsm\V_{i_1}^{\ast k_{i_1}} \cdot \mathbf{\V_{i_{l+1}}^{k_{i_{l+1}}}}\mathbf{(\idop-\V_{i_{l+1}}^{\phantom{\ast}}\!\V_{i_{l+1}}^{\ast})}\mathbf{\V_{i_{l+1}}^{\ast k_{i_{l+1}}}}.
\]
We move $\mathbf{\V_{i_{l+1}}^{k_{i_{l+1}}}}$ to the left of $\V_{i_l}^{\ast k_{i_l}}\dotsm\V_{i_1}^{\ast k_{i_1}}$ at the cost of a unimodular constant that can be determined from \cref{eq:relations_for_operators}. Since the indices are all different, \cref{res:commuting_operators} shows that it can then freely be moved further to the left of $ (\idop-\V_{i_1}^{\phantom{\ast}}\!\V_{i_1}^{\ast})\dotsm(\idop-\V_{i_l}^{\phantom{\ast}}\!\V_{i_l}^{\ast})$. We thus see that the product equals
\begin{align*}
\zz_{i_1 i_{l+1}}^{k_{i_1} k_{i_{l+1}}}&\dotsm\zz_{i_l i_{l+1}}^{k_{i_l} k_{i_{l+1}}}
\V_{i_1}^{k_{i_1}}\dotsm \V_{i_l}^{k_{i_l}} \mathbf{\V_{i_{l+1}}^{k_{l+1}}} (\idop-\V_{i_1}^{\phantom{\ast}}\!\V_{i_1}^{\ast})\dotsm (\idop -\V_{i_l}^{\phantom{\ast}}\!\V_{i_l}^{\ast}) \V_{i_l}^{\ast k_{i_l}}\dotsm \V_{i_1}^{\ast k_{i_1}}\\
&\cdot \mathbf{(\idop-\V_{i_{l+1}}^{\phantom{\ast}}\!\V_{i_{l+1}}^{\ast})}\mathbf{\V_{i_{l+1}}^{\ast k_{l+1}}}.
\end{align*}

Again since the indices are all different, \cref{res:commuting_operators} implies that the factor $\mathbf{(\idop-\V_{i_{l+1}}^{\phantom{\ast}}\!\V_{i_{l+1}}^{\ast})}$ can subsequently freely be moved to the left of $\V_{i_l}^{\ast k_{i_l}}\dotsm \V_{i_1}^{\ast k_{i_1}}$. After that, all that remains to be done is move $\mathbf{\V_{i_{l+1}}^{\ast k_{l+1}}}$ to the left of the now preceding sub-product $\V_{i_l}^{\ast k_{i_l}}\dotsm \V_{i_1}^{\ast k_{i_1}}$. This introduces a second unimodular constant, but part~\partref{part:implied_relations_3} of \cref{res:implied_relations} and the fact that $\z_{ij}=\zz_{ji}$ show that this second constant is the complex conjugate of the earlier constant. This completes the proof of \cref{eq:moving_factors_1}.

The proof of \cref{eq:moving_factors_2} is also by induction. In this case, one need merely note that, since the indices are all different, \cref{res:commuting_operators} shows that an extra factor $\V_{i_{l+1}}^{k_{l+1}}\V_{i_{l+1}}^{\ast k_{l+1}}$ commutes with the preceding factor $\V_{i_l}^{\ast k_{i_l}}\dotsm\V_{i_1}^{\ast k_{i_1}}$ that arises from the induction hypothesis.
\end{proof}

After these algebraic preparations, we can now proceed towards the first part of the space decomposition for \Vtuple. For each $i=1,\ldots,n$, equations~\eqref{eq:identity_decomposition}, \eqref{eq:Piso_as_series}, and \eqref{eq:Puni_as_limit} yield the decomposition
\begin{equation}\label{eq:identity_decomposition_for_i}
\idop=P_i^{\mathrm{iso}}+P_i^{\mathrm{uni}}
\end{equation}
of the identity operator, where
\[
P_i^{\mathrm{iso}}P_i^{\mathrm{uni}}=0
\]
and
\begin{equation}\label{eq:series_and_limit_for_i}
P_i^{\mathrm{iso}}=\sum_{k_i=0}^\infty \V_i^{k_i}(\idop-\V_i^{\phantom\ast}\! \V_i^\ast)\V_i^{\ast k_i}\quad,\quad P_i^{\mathrm{uni}}=  \mathrm{SOT-}\lim_{m_i\to \infty} \V_i^{m_i} \V_i^{\ast m_i}.
\end{equation}

\cref{res:commuting_operators} shows that the projections $\V_i^{k_i}(\idop-\V_i^{\phantom\ast}\! \V_i^\ast)\V_i^{\ast k_i}$ and $\V_i^{m_i} \V_i^{\ast m_i}$ in \cref{eq:series_and_limit_for_i} for a fixed index $i$ commute with $\V_j^{\phantom\ast}\!$ and $\V_j^\ast$ for all $j\neq i$. Taking the SOT-limits, we see that the projections $P_i^{\mathrm{iso}}$ and $P_i^{\mathrm{uni}}$ commute with $\V_j^{\phantom\ast}\!$ and $\V_j^\ast$ for all $j\neq i$; we know from \cref{res:space_decomposition_for_one_isometry} that they also commute with $\V_i^{\phantom\ast}\!$ and $\V_i^\ast$. Taking limits once more, it is now clear that we have $2n$ pairwise commuting projections $P_1^{\mathrm{iso}},\dotsc, P_n^{\mathrm{iso}},P_1^{\mathrm{uni}},\dotsc,P_n^{\mathrm{uni}}$, and that all of these commute with all  $\V_i^{\phantom\ast}\!$ and $\V_i^\ast$.

The following first part of the decomposition of the space is now a consequence of elementary manipulations with  commuting projections, combined with  \cref{res:unitary_reducing_subspaces,res:pure_isometry_reducing_subspaces}. For $n=1$, it reproduces \cref{res:space_decomposition_for_one_isometry,res:pure_isometry_reducing_subspaces,res:unitary_reducing_subspaces}.


\begin{theorem}[Space decomposition according to types of actions]\label{res:space_decomposition_types_of_actions} Let \Vtuple\ be an $n$-tuple of \dnci.
For every \uppars{possibly empty} set $A\subseteq \{1,\dotsc,n\}$ of indices, set
\begin{equation}\label{eq:A_projection_1}
\begin{split}
P_A^{{\mathrm{iso}}}&\coloneqq\prod_{i\in A} P_i^{{\mathrm{iso}}},\\
P_{A^\upc}^{\mathrm{uni}}&\coloneqq\prod_{i\in A^{\upc}} P_i^{\mathrm{uni}},\\
\end{split}
\end{equation}
\begin{equation}\label{eq:A_projection_2}
P_A\coloneqq P_A^{{\mathrm{iso}}}P_{A^\upc}^{\mathrm{uni}},
\end{equation}
and
\[
H_A\coloneqq P_A(H).
\]

Then
\[
H_A=\bigcap_{i\in A}H_{i}^{\mathrm{iso}}\,\bigcap_{i\in A^{\upc}}H_i^{\mathrm{uni}},
\]
and
\begin{equation}\label{eq:space_decomposition_types_of_actions}
H=\bigoplus_{A\subseteq \{1,\dotsc,n\}} H_A
\end{equation}
is a Hilbert space direct sum such that all summands $H_A$ reduce all $\V_i$. For all $A\subseteq \{1,\dotsc,n\}$ and $i=1,\dotsc,n$, ${\V_i}|_{H_A}$ is a pure isometry if $i\in A$, and ${\V_i}|_{H_A}$ is unitary if $i\in A^{\upc}$.

Furthermore, if $L$ is a subspace of $H$ that reduces all operators $\V_1,\dotsc, \V_n$ and if $A\subseteq \{1,\dotsc,n\}$, then the following are equivalent:
\begin{enumerate}
	\item\label{part:type_of_actions_1} ${\V_i}|_L$ is a pure isometry for all $i\in A$, and  unitary for all $i\in A^{\upc}$;
    \item\label{part:type_of_actions_2} $L\subseteq H_A$;
    \item\label{part:type_of_actions_3} $P_LP_A=P_L$.
\end{enumerate}

\end{theorem}

The remarks preceding the theorem show that the order of the factors in \cref{eq:A_projection_1,eq:A_projection_2} is immaterial, and that all these products commute with each other and with all $\V_i^{\phantom{\ast}}\!$ and $V_i^\ast$.

\begin{proof}
It is clear from commutativity that $H_A=\bigcap_{i\in A}H_{i}^{\mathrm{iso}}\,\bigcap_{i\in A^{\upc}}H_i^{\mathrm{uni}}$.

Taking the product of \cref{eq:identity_decomposition_for_i} over all indices $1,\dotsc,n$, we see that we have a decomposition
\[
\idop=\prod_{i=1}^n \lrpars{P_i^{\mathrm{iso}}+P_i^{\mathrm{uni}}}=\sum_{A\subseteq\{1,\dotsc,n\}}P_A^{\mathrm{iso}}P_{A^\upc}^{\mathrm{uni}}=\sum_{A\subseteq\{1,\dotsc,n\}}P_A
\]
of the identity operator into $2^n$ projections. Each summand corresponds to a combination of choices for either $P_i^{\mathrm{iso}}$ or $P_i^{\mathrm{uni}}$ for each $i=1,\dotsc,n$ when expanding the product, where $P_i^{\mathrm{iso}}$ has been chosen for $i\in A$, and $P_i^{\mathrm{uni}}$ for $i\in A^{\upc}$. If $A$ and $A^\prime$ are different sets of indices, then $P_AP_{A^\prime}$ involves a factor $P_i^{\mathrm{iso}}P_i^{\mathrm{uni}}$ for some $i$. Since this is zero, $H_A$ and $H_{A^\prime}$ are then orthogonal.

Since all $P_A$ commute with all $\V_i$, all subspaces $H_A$ reduce all $V_i$.

If $i\in A$, then $P_A$ contains a factor $P_i^{\mathrm{iso}}$, so that $P_A P_i^{\mathrm{iso}}=P_A$. Similarly, $P_A P_i^{\mathrm{uni}}=P_A$ if $i\in A^{\upc}$. Hence \cref{res:pure_isometry_reducing_subspaces,res:unitary_reducing_subspaces} show that $\V_i$ is a pure isometry on $H_A$ if $i\in A$ and unitary if $i\in A^{\upc}$.

It remains to establish the equivalence of the statements concerning a reducing subspace $L$.

The equivalence of \partref{part:type_of_actions_2} and \partref{part:type_of_actions_3} is clear.

We prove that \partref{part:type_of_actions_1} implies \partref{part:type_of_actions_3}. We know from  \cref{res:unitary_reducing_subspaces,res:pure_isometry_reducing_subspaces} that
$P_LP_i^{{\mathrm{iso}}}=P_L$ for all $i\in A$, and that $P_LP_i^{{\mathrm{uni}}}=P_L$ for all $i\in A^{\upc}$ . Then clearly $P_LP_A=P_L$.

We prove that  \partref{part:type_of_actions_3} implies \partref{part:type_of_actions_1}.
If $i\in A$, then $P_A^{\mathrm{iso}}P_{i}^{\mathrm{iso}}=P_A^{\mathrm{iso}}$, since $P_A^{\mathrm{iso}}$ contains a factor $P_{i}^{\mathrm{iso}}$. Therefore, we infer from $P_L P_A=P_L$ and commutativity that $P_LP_{i}^{\mathrm{iso}}=P_LP_AP_{i}^{\mathrm{iso}}=P_LP_A^{\mathrm{iso}}P_{A^\upc}^{\mathrm{uni}}P_{i}^{\mathrm{iso}}=P_LP_A^{\mathrm{iso}}P_{i}^{\mathrm{iso}}P_{A^\upc}^{\mathrm{uni}}=P_LP_A^{\mathrm{iso}}P_{A^\upc}^{\mathrm{uni}}=P_LP_A=P_L$.
Likewise, $P_LP_{i}^{{\mathrm{uni}}}=P_L$ for $i\in A^{\upc}$.
Hence \partref{part:type_of_actions_1} follows from \cref{res:unitary_reducing_subspaces,res:pure_isometry_reducing_subspaces}.
\end{proof}

The following generalisation of \cref{res:pure_isometry_and_unitary_splitting} is clear from \cref{res:space_decomposition_types_of_actions}.

\begin{proposition}\label{res:general_pure_isometry_and_unitary_splitting}\quad
\begin{enumerate}
	\item Let $A,B\subseteq \{1,\dotsc,n\}$ with $A\neq B$. Suppose that $L_A$ and $L_B$ are subspaces reducing all $\V_1,\dotsc,\V_n$ and such that ${\V_i}|_{L_A}$ is a pure isometry for $i\in A$,  ${\V_i}|_{L_A}$ is unitary for $i\in A^{\upc}$,  ${\V_i}|_{L_B}$ is a pure isometry for $i\in B$,  and ${\V_i}|_{L_B}$ is unitary for $i\in B^{\upc}$. Then $L_A\ \bot\ L_B$.
    \item Suppose that, for each $A \subseteq \{1,\dotsc,n\}$, $L_A$ is a subspace that reduces all $\V_1,\dotsc,\V_n$ and such that ${\V_i}|_{L_A}$ is a pure isometry for $i\in A$ and ${\V_i}|_{L_A}$ is unitary for $i\in A^\upc$. If $H=\bigoplus_{A\subseteq \{1,\dotsc,n\}} L_A$ \uppars{algebraically}, then $L_A=H_A$ for all $A$.
\end{enumerate}
\end{proposition}

The second part of the space decomposition is a decomposition of each $H_A$. This is obtained by inserting the right hand sides in \cref{eq:series_and_limit_for_i} into the product that is $P_A$. This will involve manipulations with limits in the strong operator topology, and we make a few preparatory remarks for this.

Firstly, if $\desset{Q_i : i\in I}$ is a countable collection of pairwise orthogonal projections, then it is easy to see that the series $\sum_i Q_i$ converges in the strong operator topology independent of the order of summation. In fact, one can partition the index set as one sees fit, sum over these (finite or infinite) subsets in any order, and then sum these partial sums in any order. The outcome is always the supremum of the $Q_i$. This implies that, in particular, multiple (countable) summations of such projections can be summed in the strong operator topology in any order.

Secondly, if $(Q_n)_{n=1}^\infty$ and $(Q^\prime_{n^\prime})_{n^\prime=1}^\infty$ are two decreasing sequences of projections with infimum $Q$ and $Q^\prime$, respectively, and such that all $Q_n$ commute with all $Q_{n^\prime}^\prime$, then one readily sees that the net $(Q_n^{\phantom\prime} Q^\prime_{n^\prime})_{(n,n^\prime)\in\NN^2}$ (with the product ordering on $\NN\times\NN$) is decreasing, and that its infimum is $Q Q^\prime$. The analogous statement holds for an arbitrary finite termwise product of such sequences.

With this in mind, we can now establish our next result. It defines and uses a subspace $W_A$ of $H_A$ that we shall call a wandering subspace; see \cref{def:wandering_subspace_and_core}.
If $z_{ij}=1$ for all $i$ and $j$ the first part of the theorem can be found as \cite[Theorem~3.1]{sarkar:2014}.
We emphasise, however, that $W_A$ is \emph{not} the analogue of the wandering subspace in \cite[p.~292]{sarkar:2014}, which is $\bigcap_{i\in A}\ker \V_i^\ast$; the reader can also compare \cref{res:space_decomposition_for_given_types_of_actions} and \cite[equation~(3.2)]{sarkar:2014}. Our subspace $W_A$ acts as a core for $H_A$ on which the operators corresponding to indices in $A^\upc$ (if any) act unitarily, and that is moved around isometrically in $H_A$ by the operators corresponding to the indices in $A$ (if any). $H_A$ is then the Hilbert direct sum of all these copies. If $W_A$ is not moved because $A=\emptyset$, then this means that $W_A$ and $H_A$ coincide. We believe that the sequel, in which our wandering subspaces play a crucial role,  shows that the definition in the present paper is the appropriate one. We shall give a conceptual characterisation of $W_A$ in \cref{res:wandering_subspace_characterisation}.

\begin{theorem}[Space decomposition for given types of actions]\label{res:space_decomposition_for_given_types_of_actions} Let \Vtuple be an $n$-tuple of \dnci.
Suppose that $A=\{i_1,\dotsc,i_l\}\subseteq \{1,\dotsc,n\}$ is a \uppars{possibly empty} set of $l$ different indices, with $A^{\upc}=\{j_1,\dotsc,j_{n-l}\}$. Set
\begin{equation*}
W_A\coloneqq\bigcap_{m_{j_1},\dotsc,m_{j_{n-l}}=0}^\infty \V_{j_1}^{m_{j_1}}\dotsm \V_{j_{n-l}}^{m_{j_{n-l}}}\Bigpars{\bigcap_{i\in A}\ker \V_i^\ast}.
\end{equation*}
Then $W_A\subseteq H_A$, and
\begin{equation*}
H_A=\bigoplus_{k_{i_1},\dotsc, k_{i_l}=0}^\infty\V_{i_1}^{k_{i_1}}\dotsm \V_{i_l}^{k_{i_l}}\lrpars{W_A}
\end{equation*}
as a Hilbert direct sum.

Here, if $A=\emptyset$, then these equations should be read as
\[
W_\emptyset\coloneqq\bigcap_{m_{j_1},\dotsc,m_{j_{n}}=0}^\infty \V_{j_1}^{m_{j_1}}\dotsm \V_{j_{n}}^{m_{j_{n}}}(H)
\]
and
\[
H_\emptyset=W_\emptyset,
\]
and, if $A=\{1,\dotsc,n\}$, then these equations should be read as
\[
W_{\{1,\dotsc,n\}}\coloneqq \bigcap_{i\in {\{1,\dotsc,n\}}}\ker \V_i^\ast
\]
and
\[
H_{\{1,\dotsc,n\}}=\bigoplus_{k_{i_1},\dotsc, k_{i_n}=0}^\infty\V_{i_1}^{k_{i_1}}\dotsm \V_{i_n}^{k_{i_n}}\lrpars{W_{\{1,\dotsc,n\}}}.
\]

Furthermore:
\begin{enumerate}
\item\label{part:WA_unitary_action} For all $i\in A^{\upc}$, $W_A$ reduces $\V_i$, and $\V_i |_{W_A}$ is unitary;
\item\label{part:WA_core_relations} We have $(\V_i|_{W_A})^\ast (\V_j|_{W_A})=\zz_{ij}(\V_j|_{W_A})(\V_i|_{W_A})^\ast$ for all $i,j\in A^{\upc}$ such that $i\neq j$;
\item\label{part:WA_zero_operator} For all $i\in A$, $V_i^{\ast}|_{W_A}=\zeop$;
\item\label{part:WA_wandering} For all $r\in\{1,\dotsc,l\}$ and $k_{i_1},\dotsc, k_{i_l}\geq 0$,
\[
\V_{i_r}^{\phantom{k_{i_r}}}\!\!\!\!\lrpars{\V_{i_1}^{k_{i_1}}\dotsm \V_{i_r}^{k_{i_r}}\dotsm\V_{i_l}^{k_{i_l}}\lrpars{W_A}}=\V_{i_1}^{k_{i_1}}\dotsm \V_{i_r}^{k_{i_r}+1}\dotsm\V_{i_l}^{k_{i_l}}\lrpars{W_A}.
\]
\end{enumerate}
\end{theorem}

\begin{definition}\label{def:wandering_subspace_and_core}
The (possibly zero) subspace $W_A$ in \cref{res:space_decomposition_for_given_types_of_actions} will be called the \emph{$A$-wandering subspace of \Vtuple}. We may list the indices in $A^\upc$ in increasing order as $j_1<\dotsb<j_{n-l}$. In that case, we shall refer to the $(|A^\upc|+1)$-tuple $(\idop_{W_A},\V_{j_1}|_{W_A}\dotsc,\V_{j_{n-l}}|_{W_A})$ as the \emph{$A$-wandering data of $(V_1,\dotsc,V_n)$}, and we shall denote it by $\wdata_{A}$; the obvious convention is that $\wdata_{\{1,\dotsc,n\}}=(\idop_{W_A})$. 
\end{definition}

\begin{remark}\label{rem:increasing_index}
If one so wishes, one can renumber the $\V_i$ in any order and place them in a new $n$-tuple of \dnci\ with permuted structure constants. The space $H_A$, however, does not depend on the numbering of the $\V_i$, but only on the \emph{set} of operators $\desset{\V_i : i\in A}$. This follows from the fact that all factors in \cref{eq:A_projection_1,eq:A_projection_2} commute. Likewise, the $A$-wandering subspace $W_A$ and the set of summands in the decomposition of $H_A$ in \cref{res:space_decomposition_for_given_types_of_actions} depend on the set $\desset{\V_i : i\in A}$ but not on the numbering; this is a consequence of the fact that the $V_i$ commute up to non-zero scalars.
In view of all this, it seems perhaps more natural to define the $A$-wandering data  not as a tuple but as the set $\{\idop_{W_A},\V_{j_1}|_{W_A}\dotsc,\V_{j_{n-l}}|_{W_A}\}$, which would then also be independent of the numbering. In that case, however, if $i\in A^\upc$, then the link between $\V_i$ and its restriction $\V_i|_{W_A}$ would be lost. All one would know is that this restriction is `somewhere' in the set of $A$-wandering data. This is undesirable when considering unitary equivalence of $n$-tuples in the sequel. It is for this reason that we insist on keeping our numbering of the $\V_i$ fixed and listing the restricted operators in the wandering data in order of increasing index. This ensures that it is always still possible to couple the original operator and its restriction to the wandering subspace.
\end{remark}

\begin{proof}[Proof of \cref{res:space_decomposition_for_given_types_of_actions}]
We shall give the proof if $l$ is such that $1\leq l\leq n-1$. The proofs for the remaining cases where $l=0$ or $l=n$ are similar and somewhat easier; they are left to the reader.

We start by proving that each $H_A$ is the Hilbert direct sum as stated.

We have
\begin{equation*}
\begin{split}
P_A&=P_A^{\mathrm{iso}}P_{A^{\upc}}^{\mathrm{uni}}\\
&=\lrpars{\prod_{i\in A}P_i^{\mathrm{iso}}}\cdot P_{A^\upc}^{\mathrm{uni}}\\
&=\left(\sum_{k_{i_1}=0}^{\infty} \V_{i_1}^{k_{i_1}}(\idop -\V_{i_1}^{\phantom\ast}\!\V_{i_1}^{\ast })\V_{i_1}^{\ast k_{i_1}}\right)\dotsm\left(\sum_{k_{i_l}=0}^{\infty} \V_{i_l}^{k_{i_l}}(\idop -\V_{i_l}^{\phantom\ast}\!\V_{i_l}^{\ast})\V_{i_l}^{\ast k_{i_l}}\right)\cdot P_{A^\upc}^{\mathrm{uni}}.
\end{split}
\end{equation*}
Within each series, the summands are pairwise orthogonal projections. Since these summands commute with all summands of the other series by \cref{res:commuting_operators}, and also with $P_{A^\upc}^{\mathrm{uni}}$, we see that we can write
\begin{equation}\label{eq:P_A_as_series}
P_A=\sum_{k_{i_1},\dotsc, k_{i_l}=0}^\infty \biggpars{\lrbrackets{\V_{i_1}^{k_{i_1}}(\idop -\V_{i_1}^{\phantom\ast}\!\V_{i_1})\V_{i_1}^{\ast k_{i_1}}}\dotsm\lrbrackets{\V_{i_l}^{k_{i_l}}(\idop -\V_{i_l}^{\phantom\ast}\!\V_{i_l}^{\ast})\V_{i_l}^{\ast k_{i_l}}}\cdot P_{A^\upc}^{\mathrm{uni}}}
\end{equation}
as an SOT-convergent series of pairwise orthogonal projections, the ranges of which are then contained in the range of $P_A$, i.e.\ in $H_A$. Hence the proof of the decomposition of $H_A$ as a Hilbert direct sum will be complete when we show that the summands in the decomposition correspond to the images of the projection summands in \cref{eq:P_A_as_series}.

For this, fix a projection summand
\begin{equation*}
\lrbrackets{\V_{i_1}^{k_{i_1}}(\idop -\V_{i_1}^{\phantom\ast}\!\V_{i_1}^{\ast})\V_{i_1}^{\ast k_{i_1}}}\dotsm\lrbrackets{\V_{i_l}^{k_{i_l}}(\idop -\V_{i_l}^{\phantom\ast}\!\V_{i_l}^{\ast})\V_{i_l}^{\ast k_{i_l}}}\cdot P_{A^\upc}^{\mathrm{uni}}.
\end{equation*}
We apply \cref{eq:moving_factors_1} and the fact that $P_{A^\upc}^{\mathrm{uni}}$ commutes with all $\V_i$ to see that this projection summand equals
\begin{equation}\label{eq:summand}
\V_{i_1}^{k_{i_1}}\dotsm\V_{i_l}^{k_{i_l}}\Bigpars{P_{A^\upc}^{\mathrm{uni}}(\idop -\V_{i_1}^{\phantom\ast}\!\V_{i_1}^{\ast})\dotsm(\idop -\V_{i_l}^{\phantom\ast}\!\V_{i_l}^{\ast})} \V_{i_l}^{\ast k_{i_l}}\dotsm\V_{i_1}^{\ast k_{i_1}}.
\end{equation}
Note that $(\idop -\V_{i_1}^{\phantom\ast}\!\V_{i_1}^{\ast})\dotsm(\idop -\V_{i_l}^{\phantom\ast}\!\V_{i_l}^{\ast})$ is a product of commuting projections. Hence it is the projection onto the intersection of their images $\bigcap_{i\in A}\ker \V_i^\ast$. We denote this projection by $Q_A$ for short.

We shall now first identify the factor $P_{A^\upc}^{\mathrm{uni}}Q_A$ in the middle of \cref{eq:summand}, and for this we proceed as follows. Note that $P_{A^\upc}^{\mathrm{uni}}$ is the infimum of the decreasing net
\[
\lrpars{\V_{j_1}^{m_{j_1}}\V_{j_1}^{\ast m_{j_1}}\dotsm\V_{j_{n-l}}^{m_{j_{n-l}}}\V_{j_{n-l}}^{\ast m_{j_{n-l}}}}_{(m_{j_1},\dotsc,m_{j_{n-l}})\in\NN^{n-l}}.
\]
Since $Q_A$ commutes with all elements of this net by \cref{res:commuting_operators} (there is no overlap in indices between $A$ and $A^\upc$), the net
\[
\lrpars{\V_{j_1}^{m_{j_1}}\V_{j_1}^{\ast m_{j_1}}\dotsm\V_{j_{n-l}}^{m_{j_{n-l}}}\V_{j_{n-l}}^{\ast m_{j_{n-l}}}Q_A}_{(m_{j_1},\dotsc,m_{j_{n-l}})\in\NN^{n-l}}
\]
is again decreasing, and its infimum is $P_{A^\upc}^{\mathrm{uni}}Q_A$. \Cref{eq:moving_factors_2} and again \cref{res:commuting_operators} show that the latter net can be rewritten as
\[
\lrpars{\V_{j_1}^{m_{j_1}}\dotsm\V_{j_{n-l}}^{m_{j_{n-l}}} Q_A \V_{j_{n-l}}^{\ast m_{j_{n-l}}}\dotsm\V_{j_1}^{\ast m_{j_1}}}_{(m_{j_1},\dotsc,m_{j_{n-l}})\in\NN^{n-l}}.
\]
In this form we can recognise the elements of this net: they are the range projections of the partial isometries $\V_{j_1}^{m_{j_1}}\dotsm\V_{j_{n-l}}^{m_{j_{n-l}}} Q_A$. That is, they are the projections onto $\V_{j_1}^{m_{j_1}}\dotsm\V_{j_{n-l}}^{m_{j_{n-l}}}\lrpars{\bigcap_{i\in A}\ker \V_i^\ast}$.  But then $P_{A^\upc}^{\mathrm{uni}}Q_A$, being the infimum of the net, is the projection onto
\[
\bigcap_{m_{j_1},\dotsc,m_{j_{n-l}}=0}^\infty \V_{j_1}^{m_{j_1}}\dotsm \V_{j_{n-l}}^{m_{j_{n-l}}}\Bigpars{\bigcap_{i\in A}\ker \V_i^{\ast}},
\]
which is $W_A$.

Now that we have identified the range of the projection that is the middle factor $P_{A^\upc}^{\mathrm{uni}}(\idop -\V_{i_1}^{\phantom\ast}\!\V_{i_1}^{\ast})\dotsm(\idop -\V_{i_l}^{\phantom\ast}\!\V_{i_l}^{\ast})$ in the projection summand in \cref{eq:summand}, we can use a similar argument to see that the projection summand as a whole is the projection onto
\[
\V_{i_1}^{k_{i_1}}\dotsm\V_{i_l}^{k_{i_l}}\left(W_A\right).
\]
The required correspondence between the images of the projection summands in \cref{eq:P_A_as_series} and the summands of $H_A$ in the statement of the theorem has now been established. Since we know that these summands are all subspaces of $H_A$, this is, in particular, true for $W_A$.

We turn to the remaining statements.

As we have seen above, the projection onto $W_A$ is $P_AQ_A$,  for which we have the factorisation $P_AQ_A=P_{A^\upc}^{\mathrm{uni}}(\idop -\V_{i_1}^{\phantom\ast}\!\V_{i_1}^{\ast})\dotsm(\idop -\V_{i_l}^{\phantom\ast}\!\V_{i_l}^{\ast})$. Since $P_{A^\upc}^{\mathrm{uni}}$ commutes with all $\V_i$, and $Q_A$ commutes with $\V_{j_1},\dotsc,\V_{j_l}$ according to \cref{res:commuting_operators}, $P_A Q_A$ commutes with $\V_{j_1},\dotsc,\V_{j_l}$.  Hence $W_A$ reduces these operators. Since $W_A\subseteq H_A$, and $H_A$ is contained in each of $H_{j_1}^{\mathrm{uni}},\dotsc,H_{j_{n-l}}^{\mathrm{uni}}$, \cref{res:unitary_reducing_subspaces} shows that $\V_{j_1},\dotsc,\V_{j_{n-l}}$ are all unitary on $W_A$. We have thus established  part~\partref{part:WA_unitary_action}.

Part~\partref{part:WA_core_relations} is evident since the relations as operators on $H$ are inherited by their restrictions to reducing subspaces.

Part~\partref{part:WA_zero_operator} is clear once one realises that all isometries and their adjoints commute up to scalars, so that, in particular, this the case for the $\V_i^\ast$ for $i\in A$ on the one hand, and $\V_{j_1},\dotsc,\V_{j_{n-l}}$ on the other hand.

Part~\partref{part:WA_wandering} follows likewise from the fact that the $\V_{i_1},\dotsc,\V_{i_l}$ commute up to non-zero constants.

\end{proof}

The following result gives a conceptual characterisation of the subspace $W_A$ in \cref{res:space_decomposition_for_given_types_of_actions}. If $n=1$ and $A=\{1\}$ it coincides with the familiar result that there is only one wandering subspace for an isometry $\S$, namely, $\ker \S^\ast$.

\begin{proposition}\label{res:wandering_subspace_characterisation}
Let $A=\{i_1,\dotsc,i_l\}\subseteq \{1,\dotsc,n\}$ be a non-empty set of $l$ different indices, with $A^\upc=\{j_1,\dotsc,j_{n-l}\}$. Suppose that $L$ is a subspace of $H_A$ that is invariant under $V_{j_1},\dotsc,V_{j_{n-l}}$ and such that
\begin{equation}\label{eq:wandering_decomposition}
H_A=\bigoplus_{k_{i_1},\dotsc, k_{i_l}=0}^\infty\V_{i_1}^{k_{i_1}}\dotsm \V_{i_l}^{k_{i_l}}(L)
\end{equation}
as a Hilbert direct sum. Then $L=W_A$ as in \cref{res:space_decomposition_for_given_types_of_actions}.
\end{proposition}

The case where $A=\emptyset$ has been left out, because then the interpretation of the Hilbert direct sum in the statement becomes unclear. Conceptually, this case is still included: if $L\subseteq H_\emptyset$ is such that $H_\emptyset$ is the Hilbert direct sums of $L$ and all its images under the operators corresponding to the indices in $A$ (of which there are none), then trivially $L=H_\emptyset$. Together with the definition of $W_\emptyset$ in \cref{res:space_decomposition_for_given_types_of_actions} this shows that $L=W_\emptyset$.

\begin{proof}
Let $r\in\{1,\dotsc,l\}$. Since the isometries commute up to non-zero constants, we see from an application of $V_{i_r}$ to \cref{eq:wandering_decomposition} that $V_{i_r}(H_A)$ is a Hilbert direct sum of summands that already occur in the right hand side of \cref{eq:wandering_decomposition}. The summand $L$, however, is no longer present, and this shows that $L$ and $\V_{i_r}(H_A)$ are orthogonal. Since $L$, being a subspace of $H_A$, is orthogonal to the spaces $H_{A^\prime}$ for all $A^\prime\neq A$, and since these spaces are invariant under $\V_{i_r}$, we see that $L$ is orthogonal to $V_{i_r}(H)$. That is, $L\subseteq \ker \V_{i_r}^\ast$. This shows that $L\subseteq\bigcap_{i\in A}\ker \V_i^\ast$.

If $A=\{1,\ldots,n\}$, then this means that $L\subseteq W_A$, where $W_A$ is as in \cref{res:space_decomposition_for_given_types_of_actions}. A comparison of \cref{eq:wandering_decomposition} and \cref{res:space_decomposition_for_given_types_of_actions} now shows that we cannot have a proper inclusion $L\subsetneq W_A$. Hence $L=W_A$, as required.

If $A\subsetneq\{1,\dotsc,n\}$, we need to continue.

Let $r\in\{1,\dotsc,{n-l}\}$. Again since the isometries commute up to non-zero constants, we see from an application of $\V_{j_r}$ to \cref{eq:wandering_decomposition} that
\begin{equation*}
\V_{j_r}(H_A)=\bigoplus_{k_{i_1},\dotsc, k_{i_l}=0}^\infty\V_{i_1}^{k_{i_1}}\dotsm \V_{i_l}^{k_{i_l}}(\V_{j_r}(L)).
\end{equation*}
Since we know from \cref{res:space_decomposition_types_of_actions}  that $\V_{j_r}(H_A)=H_A$, and $\V_{j_r}(L)\subseteq L$ by hypothesis, a comparison with \cref{eq:wandering_decomposition} shows that there cannot be a proper inclusion $\V_{j_r}(L)\subsetneq L$. Hence
$L=\V_{j_r}(L)$. Combining this with $L\subseteq\bigcap_{i\in A}\ker \V_i^\ast$, we see that $L\subseteq W_A$, where $W_A$ is as in \cref{res:space_decomposition_for_given_types_of_actions}. Now that we know this, a comparison of \cref{eq:wandering_decomposition} and \cref{res:space_decomposition_for_given_types_of_actions}  shows that we cannot have a proper inclusion $L\subsetneq W_A$. Hence $L=W_A$, as desired.
\end{proof}

We include the following inheritance result.

\begin{proposition}\label{res:commuting_operator_is_reduced_several_isometries}
Suppose that $T\in\bounded(H)$ commutes with the projections $\V_i^k\V_i^{\ast k}$ for all $i=1,\dotsc,n$ and all $k\geq 1$; equivalently, suppose that $\V_i^k(H)$ reduces $T$ for all $i=1,\dotsc,n$ and all $k\geq 1$. Then, for all \uppars{possibly empty} $A\subseteq\{1,\dotsc,n\}$, all summands in the decomposition of $H_A$ as a Hilbert direct sum in \cref{res:space_decomposition_for_given_types_of_actions} reduce $T$; in particular this is the case for the $A$-wandering subspace $W_A$. Consequently, $H_A$ reduces $T$.
\end{proposition}

\begin{proof}
In view of the definition of $P_{A^\upc}^{\mathrm{uni}}$ in \cref{eq:A_projection_1}, and of the $P_i^{\mathrm{uni}}$ in \cref{eq:Puni_as_limit}, he hypothesis evidently implies that the projection summands in the proof of \cref{res:space_decomposition_for_given_types_of_actions} commute with $T$.
\end{proof}

If all $\z_{ij}$ are equal to 1, if $T$ is an isometry that commutes with all $\V_i^{\phantom\ast}\!$ and $\V_i^\ast$, and if $A=\{1,\dotsc,n\}$, then \cref{res:commuting_operator_is_reduced_several_isometries} yields \cite[Proposition~2.2]{sarkar:2014}.

We conclude this section with an application.

\begin{lemma}\label{res:surjectivity}
Suppose that \Vtuple\ is an $n$-tuple of \dnci\ where $\V_1,\dotsc,\V_{n}$ are all pure isometries, and that $\bigcap_{i=1}^n \ker \V_i^\ast$ has finite dimension. Let $T\in \bounded(H)$ and suppose that, for $i=1,\dotsc,n$,  $T\V_i=\tau_i\V_i T$ for some $\tau_i\in\TT$, and $T\ker \V_i^\ast\subseteq \ker \V_i^\ast$.

If $T$ has trivial kernel on $\bigcap_{i=1}^n \ker \V_i^\ast$, then $T$ maps $H$ onto $H$.
\end{lemma}

\begin{proof}
Since $T$ leaves $\bigcap_{i=1}^n \ker \V_i^\ast$ invariant and has trivial kernel on this finite dimensional space, we see that $T:\bigcap_{i=1}^n \ker \V_i^\ast\mapsto\bigcap_{i=1}^n \ker \V_i^\ast$ is a bijection. We let $T^{-1}$ denote its inverse on this subspace, which is automatically bounded.

By \cref{res:space_decomposition_types_of_actions,res:space_decomposition_for_given_types_of_actions} we have
\[
H=\bigoplus_{k_1,\dotsc,k_n=0}^\infty \V_1^{k_1}\dotsm \V_{n}^{k_n} \Bigpars{\bigcap_{i=1}^n\ker \V_i^\ast}.
\]

Let $x\in H$. Then we can write
\begin{equation}\label{eq:x_given}
x=\sum_{k_1,\dotsc,k_n=0}^\infty \V_1^{k_1}\dotsm \V_{n}^{k_n} x_{k_1,\dotsc,k_n}
\end{equation}
as an orthogonal series, where  $x_{k_1,\dotsc,k_n}\in \bigcap_{i=1}^n \ker \V_i^\ast$ and $\sum_{k_1,\dotsc,k_n=0}^\infty \lrnorm{x_{k_1,\dotsc,k_n}}^2=\lrnorm{x}^2$. It is clear from the relations between $T$ and the $V_i$ that
\[
Tx=\sum_{k_1,\dotsc,k_n=0}^\infty c_{k_1,\dotsc,k_n} \V_1^{k_1}\dotsm \V_{n}^{k_n} Tx_{k_1,\dotsc,k_n}
\]
is then again an orthogonal series, where the $c_{k_1,\dotsc,k_n}$ are unimodular constants.

We combine the above: if $x\in H$ is as in \cref{eq:x_given}, then, since $T^{-1}$ is bounded,
\[
\sum_{k_1,\dotsc,k_n=0}^\infty c_{k_1,\dotsc,k_n}^{-1} \V_1^{k_1}\dotsm \V_{n}^{k_n} T^{-1}x_{k_1,\dotsc,k_n}
\]
is a convergent orthogonal series. If $y$ denotes its sum, then $Ty=x$.
\end{proof}

The following is an immediate consequence. It it conceivable that a proof can be given that avoids the use of our results so far, but without these it might be hard to spot the result at all.

\begin{corollary}\label{res:remaining_ones_are_unitaries}
Let $\Vtuple$ be an $n$-tuple of \dnci. Let $l$ be such that $1\leq l\leq n$. If $\V_1,\dotsc, \V_l$ are pure isometries such that $\dim(\bigcap_{i=1}^l\ker \V_i^\ast)<\infty$, then $\V_{l+1},\dotsc, \V_n$ are unitary operators on $H$.
\end{corollary}

\begin{proof}
This follows from \cref{res:surjectivity} for $(\V_1,\dotsc,\V_l)$, combined with \cref{res:implied_relations}.
\end{proof}

As a particular case, if $S$ is the unilateral shift on $\ell^2(\NN_0)$, and $\V$ is an isometry such that $S^\ast \V=z \V S ^\ast$ with $|z|=1$, then $\V$ is unitary.

Here certainly a direct proof is possible, as follows.
Since $S^\ast e_0=0$ and $\ker S^\ast=\CC e_0$, we see from the given relation that $\V e_0=\lambda e_0$ for some $\lambda\in\CC$.
Then $|\lambda|=1$ since $\V$ is an isometry.
Next, $S^\ast \V e_1=z\V S^\ast e_1=z\V e_0=\lambda z e_0$.
Hence $\V e_1=\lambda ze_1+\mu e_0$ for some $\mu\in\CC$.
Then $\V e_1=\lambda ze_1$ since $\V$ is an isometry.
Induction shows that $\V e_k=\lambda z^k e_k$ for $k\geq 0$. Hence $\V$ is unitary.


\section{Wold decomposition and examples}\label{sec:Wold_decomposition_and_examples}


In view of \cref{res:space_decomposition_types_of_actions}, if $A\subseteq\{1,\dotsc,n\}$, then we would like to know more about the structure  of an $n$-tuple of \dnci\ such that the operators corresponding to the indices in $A$ are pure isometries, and the operators corresponding to the remaining indices are unitary. At the same time, we are interested to find an example of such an $n$-tuple on a non-zero Hilbert space. We may restrict ourselves to the case where the indices in $A$ come first; this makes the notation a little less demanding.  Choosing a more suggestive notation than the generic letter $\V$, we shall, therefore, be working with an $n$-tuple $(S_1,\dotsc,S_l,U_{l+1},\dotsc,U_n)$ such that $S_1,\dotsc,S_l$ are pure isometries and $U_{l+1},\dotsc,U_n$ are unitary. Here $0\leq l\leq n$, so that one of the two lists in the $n$-tuple could be absent.
Using the results in Section~\ref{sec:space_decomposition_general_case}, we shall now analyse such $n$-tuples; this leads to a Wold decomposition. As we shall see, this decomposition informs us how to find non-zero examples. As explained in Section~\ref{sec:introduction_and_overview}, the description of irreducible tuples in \cite[Theorem~2]{proskurin_2000} could serve as an alternate source of inspiration.

We start with the case where $l=0$, i.e.\ where the list of $S_i$ is empty. In view of \cref{res:unitary}, an application of \cref{res:space_decomposition_types_of_actions,res:space_decomposition_for_given_types_of_actions} yields that $H=H_\emptyset=W_\emptyset$ and that the $U_i$ are unitary operators on $W_\emptyset$ satisfying \cref{eq:relations_for_operators}. That is merely reiterating our starting point. There does not seem much that we can add here: we are simply looking at a representation of the non-commutative $n$-torus and with this we hit rock bottom. In the terminology of \cref{def:wandering_subspace_and_core}, the $\emptyset$-wandering data  $\wdata_{\emptyset}$ of $(U_1,\dotsc,U_n)$ are  $(\idop_H, U_1,\dotsc,U_n)$.
For reasons of uniformity that will become clear below, we prefer to denote the space that the pertinent unitary operators act on by $W$, and we shall tautologically refer to such an $n$-tuple $(U_1,\dotsc,U_n)$ of \dnci\ in which all isometries are unitary operators as \emph{the standard $n$-tuple of doubly non-commutative isometries with $\emptyset$-wandering data  $(\idop_W,U_1,\dotsc,U_n)$}. As a consequence of \cref{res:unitary,res:space_decomposition_types_of_actions}, all other wandering data  are zero tuples.

It seems as if the results in Section~\ref{sec:space_decomposition_general_case} do not help to find a non-zero example for $l=0$, and that we need to refer to the literature (e.g.\ to \cite{rieffel:1990}) for these. That is not entirely true, though: we shall see how the analysis of the case where $l\geq 1$ still tells us how to construct such an example if $l=0$. As we shall see, such `fully unitary' examples are, in fact, also needed when $l\geq 1$. Since they are easier than those for the latter case and virtually immediate from that case, we defer the non-zero example where $l=0$ until the case where $l\geq 1$ has been handled.

We turn to the case where $l\geq 1$. Contrary to the case where $l=0$,  \cref{res:space_decomposition_types_of_actions,res:space_decomposition_for_given_types_of_actions} now give some new information. We start by proving a Wold decomposition for $(S_1,\dotsc,S_l,U_{l+1},\dotsc,U_n)$.

Let us first suppose that also $l\leq  n-1$, so that there are at least one pure isometry and one unitary operator in our $n$-tuple; this avoids working with conventions for empty sets of operators in the argumentation below. \cref{res:space_decomposition_types_of_actions,res:space_decomposition_for_given_types_of_actions} show that there exists a subspace $W$ of $H$ such that
\begin{equation}\label{eq:wold_towards_model}
H=\bigoplus_{k_1,\dotsc,k_l=0}^\infty S_1^{k_1}\dotsm S_l^{k_l} (W).
\end{equation}
Furthermore, $W$ is invariant under $U_{l+1},\dotsc,U_n$, and these operators all act on $W$ as unitary operators.
Writing $\widetilde{U}_i={U_i}|_W$, we have
\[
\widetilde{U}_i^\ast\widetilde{U}_j^{\phantom{\ast}}\!=\zz_{ij}\widetilde{U}_j^{\phantom{\ast}}\!\widetilde{U}_i^\ast
\]
for all $i,j=l+1,\dotsc,n$ with $i\neq j$.

\Cref{eq:wold_towards_model} enables us to define an isomorphism $\varphi: H\to \ell^2(\NN_0^l)\otimes W$, as follows. Using the natural notation for the canonical orthonormal basis of $\ell^2(\NN_0^l)$, set
\begin{equation}\label{eq:phi_definition}
\varphi\left(\sum_{k_1,\dotsc,k_l=0}^\infty S_1^{k_1}\dotsm S_l^{k_l} x_{k_1,\dotsc,k_l}\right) = \sum_{k_1,\dotsc,k_l=0}^\infty e_{k_1,\dotsc,k_l}\otimes x_{k_1,\dotsc,k_l},
\end{equation}
where the $x_{k_1,\dotsc,k_l}$ are in $W$. We note that $\norm{e_{k_1,\dotsc,k_l}\otimes x_{k_1,\dotsc,k_l}}=\norm{x_{k_1,\dotsc,k_l}}=\norm{S_1^{k_1}\dotsm S_l^{k_l} x_{k_1,\dotsc,k_l}}$, so that the convergence of the orthogonal series in the left hand side of \cref{eq:phi_definition} is equivalent with that of the orthogonal series in the right hand side. Hence $\varphi$ is indeed an isomorphism. Aside, we also note that, although the decomposition in \cref{eq:wold_towards_model} as a Hilbert direct sum is (up to a permutation of the summands) independent of the choice for the numbering of the $S_i$, this is no longer the case for the definition of $\varphi$ in \cref{eq:phi_definition}. This certainly depends on this choice. However, since it is only the existence of such $\varphi$ that we need, this will not bother us. We simply work with $\varphi$ as it is determined by the chosen and fixed enumeration of our $n$ isometries.

We can now transfer the action of our given $S_1,\dotsc,S_l$ and $U_{n-l},\ldots,U_n$ to $\ell^2(\NN_0^l)\otimes W$ via $\varphi$. It is easy to determine what these transferred actions look like. Doing so for the $S_i$, one encounters expressions of the form $S_i\cdot S_1^{k_1}\dotsm S_l^{k_l}x_{k_1,\dotsc,k_l}$, where $S_i$ needs to be moved to its `proper' place in the operator part of $S_i\cdot S_1^{k_1}\dotsm S_l^{k_l}$ of such an expression. Since $S_i$ needs to pass the powers of $S_1,\dotsc,S_{i-1}$ for this, a constant appears that involves the $(i-1)$ constants $\z_{i,1},\dotsc,\z_{i,i-1}$.
Doing so for the $U_i$, one encounters expressions of the form $U_i\cdot S_1^{k_1}\dotsm S_l^{k_l}x_{k_1,\dotsc,k_l}$. In this case, $U_i$ always needs to be moved to become the final operator in the operator part $U_i\cdot S_1^{k_1}\dotsm S_l^{k_l}$. Thus there are always $l$ constants involved, namely, $z_{i,1},\dotsc,\z_{i,l}$. After having become the rightmost operator, $U_i$ acts on $x_{k_1,\dotsc,k_l}$ as $\widetilde{U}_i$.

Thus one sees that, for $i=1,\dotsc,l$,
\begin{equation}\label{eq:classification_pure_isometry_action}
(\varphi\circ S_i\circ\varphi^{-1})(e_{k_1,\dotsc,k_l}\otimes x)=\z_{i,1}^{k_1}\dotsm\z_{i,i-1}^{k_{i-1}}\ e_{k_1, \dotsc, k_{i-1}, k_i+1, k_{i+1} ,\dotsc, k_l} \otimes x
\end{equation}
for all $k_1,\dotsc,k_l\geq 0$ and $x\in W$, and that, for $i=l+1,\dotsc,n$,
\begin{equation}\label{eq:classification_unitary_action}
(\varphi\circ U_i\circ \varphi^{-1}) (e_{k_1,\dotsc,k_l}\otimes m) = \z_{i,1}^{k_1}\dotsm\z_{i,l}^{k_{l}}\ e_{k_1,\dotsc,k_l}\otimes \widetilde{U}_i x
\end{equation}
for all $k_1,\dotsc,k_l\geq 0$ and $x\in W$, where the ${\widetilde U}_i$ are unitary operators on $W$ satisfying
\begin{equation}\label{eq:classification_unitaries_relation}
\widetilde{U}_i^\ast\widetilde{U}_j^{\phantom{\ast}}\!=\zz_{ij}\widetilde{U}_j^\ast\widetilde{U}_i^{\phantom{\ast}}\!
\end{equation}
for all $i,j=l+1,\dotsc,n$ with $i\neq j$. As elsewhere in this section, the empty products in \cref{eq:classification_pure_isometry_action} that occur for $i=1$ should be read as 1. Moving $S_1$ is never necessary.

There does not seem to be anything that can be said further. This would have to be related to the structure of $(\widetilde{U}_{l+1},\dotsc,\widetilde{U}_n)$, but, as earlier, $W$ is simply a module over the pertinent non-commutative $(n-l)$ torus and that is where it stops. We have thus obtained a Wold decomposition.

It is now also clear how examples can be obtained: turning the tables, we simply use \cref{eq:classification_pure_isometry_action,eq:classification_unitary_action,eq:classification_unitaries_relation} as an Ansatz.

Suppose, therefore, that $l$ is such that $1\leq l\leq n-1$ and that a Hilbert space $W$ is given with unitary operators $\widetilde{U}_{l+1},\dotsc,\widetilde{U}_n\in\bounded(W)$ satisfying \cref{eq:classification_unitaries_relation} for all $i,j=l+1,\dotsc,n$ with $i\neq j$. Then we introduce operators $S_1,\dotsc,S_l,U_{l+1},\dotsc,U_n$ on $\ell^2(\NN_0^l)\otimes W$, as follows.

For $i=1,\dotsc,l$, set
\begin{equation}\label{eq:definition_pure_isometry_action}
S_i(e_{k_1,\dotsc,k_l}\otimes x)\coloneqq\z_{i,1}^{k_1}\dotsm\z_{i,i-1}^{k_{i-1}}\ e_{k_1, \dotsc, k_{i-1}, k_i+1, k_{i+1} ,\dotsc, k_l} \otimes x
\end{equation}
for all $k_1,\dotsc,k_l\geq 0$ and $x\in W$, and, for $i=l+1,\dotsc,n$, set
\begin{equation}\label{eq:definition_unitary_action}
U_i(e_{k_1,\dotsc,k_l}\otimes x) \coloneqq \z_{i,1}^{k_1}\dotsm\z_{i,l}^{k_{l}}\ e_{k_1,\dotsc,k_l}\otimes \widetilde{U}_i x
\end{equation}
for all $k_1,\dotsc,k_l\geq 0$ and $x\in W$.

The $S_i$ and $U_i$ are all tensor products of operators on $\ell^2(\NN_0^l)$ and $W$. An operator $S_i$ is the tensor product of an operator on $\ell^2(\NN_0^l)$ that is the direct sum of weighted unilateral shifts (with a weight that is constant in every copy, but where the constant that is this weight depends on the copy) and the identity operator on $W$. An operator $U_i$ is the tensor product of a diagonal unitary operator on $\ell^2(\NN_0^l)$ and the unitary operator ${\widetilde U}_j$ on $W$.

It is clear from \cref{res:pure_isometry} that the $S_i$ are pure isometries, since the corresponding subspaces $H_i^{\mathrm{uni}}$ are all the zero subspace. The $U_i$ are obviously unitary. Hence it remains to verify the relations, which we shall now do.

Using that $S_i^{\ast}S_i=\idop$ and that $\ker S_i^\ast=(S_i(\ell^2(\NN_0^l)\otimes W))^\bot$, it is easy to see that
\begin{equation}\label{eq:adjoint_of_pure_isometry}
S_i^\ast (e_{k_1,\dotsc,k_l}\otimes x) =
\begin{cases}
\zz_{i,1}^{k_1}\dotsm\zz_{i,i-1}^{k_{i-1}}\ e_{k_1, \dotsc, k_{i-1}, k_i-1, k_{i+1} ,\dotsc, k_l} \otimes x& \textup{if } k_i\geq 1;\\
0& \textup{if } k_i=0.
\end{cases}
\end{equation}

It is evident from \cref{eq:definition_unitary_action} that, for $i=l+1,\dotsc,n$,
\begin{equation}\label{eq:adjoint_of_unitary_from_definition}
U_i^\ast e_{k_1,\dotsc,k_l}\otimes x = \zz_{i,1}^{k_1}\dotsm\zz_{i,l}^{k_{l}}\ e_{k_1,\dotsc,k_l}\otimes \widetilde{U}_i^\ast x
\end{equation}
for all $k_1,\dotsc,k_l\geq 0$ and $x\in W$.

\begin{lemma}\label{res:isometries_satisfy_relations}
For the pure isometries $S_1,\dotsc,S_l$ on $\ell^2(\NN_0^l)\otimes W$ as defined in \cref{eq:definition_pure_isometry_action}, we have
\[
S_i^\ast S_j= \zz_{ij} S_jS_i^\ast
\]
for all $i,j=1,\dotsc,l$ such that $i\neq j$.
\end{lemma}

\begin{proof}
If we can prove the statement when $i<j$, then the case where $i>j$ follows from taking adjoints and using that $\z_{ji}=\zz_{ij}$ whenever $i\neq j$. Hence we suppose that $i<j$.

First of all, if $k_i=0$ then $S_i^\ast S_j (e_{k_1,\dotsc,k_l}\otimes x)$ and $S_jS_i^\ast (e_{k_1,\dotsc,k_l}\otimes x)$ are both zero. This is still immediately clear for all $i$ and $j$: the reason is that $S_j$ does not increase $k_i$.  We shall use that $i<j$ for the remaining case where $k_i\geq 1$, to which we now turn. As we shall see, the factor $\zz_{ij}$ in the relation originates from the fact that the $i$-th index precedes the $j$-th index in the labelling of the $e_{k_1,\dotsc,k_l}$. Indeed,
\begin{align*}
S_i^\ast & S_j (e_{k_1,\dotsc,k_l}\otimes x)  \\
&=(\z_{j,1}^{k_1}\dotsm\z_{j,j-1}^{k_{j-1}})\cdot S_i^\ast (e_{k_1,\dotsc, k_{j-1}, k_j+1, k_{j+1} ,\dotsc, k_l} \otimes x) \\
&= (\z_{j,1}^{k_1}\dotsm\z_{j,j-1}^{k_{j-1}})\cdot (\zz_{i,1}^{k_1}\dotsm\zz_{i,i-1}^{k_{i-1}})\cdot  e_{k_1,\dotsc,k_{i-1}, k_i-1, k_{i+1},\dotsc, k_{j-1}, k_j+1, k_{j+1} ,\dotsc, k_l} \otimes x
\intertext{and}
S_j&S_i^\ast(e_{k_1,\dotsc,k_l}\otimes x)\\
& = (\zz_{i,1}^{k_1}\dotsm \zz_{i,i-1}^{k_{i-1}})\cdot S_j (e_{k_1,\dotsc, k_{i-1}, k_i-1, k_{i+1} ,\dotsc, k_l} \otimes x)\\
& = (\zz_{i,1}^{k_1}\dotsm\zz_{i,i-1}^{k_{i-1}})\cdot (\z_{j,1}^{k_1}\dotsm\z_{j,i-1}^{k_{i-1}} \mathbf{\z_{j,i}^{k_i -1}}\z_{j,i+1}^{k_{i+1}}\dotsm\z_{j,j-1}^{k_{j-1}}) \cdot \\
&\quad\quad\quad\quad\quad\quad\quad\quad\quad\quad\quad\quad\quad\quad\quad\quad
 e_{k_1,\dotsc,k_{i-1}, k_i-1, k_{i+1},\dotsc, k_{j-1}, k_j+1, k_{j+1} ,\dotsc, k_l} \otimes x\\
&=(\zz_{i,1}^{k_1}\dotsm \zz_{i,i-1}^{k_{i-1}})\cdot \mathbf{\z_{j,i}^{-1}}\cdot(\z_{j,1}^{k_1}\dotsm\z_{j,i-1}^{k_{i-1}}\mathbf{\z_{j,i}^{k_i}}\z_{j,i+1}^{k_{i+1}}\dotsm\z_{j,j-1}^{k_{j-1}})\cdot\\
&\quad\quad\quad\quad\quad\quad\quad\quad\quad\quad\quad\quad\quad\quad\quad\quad
  e_{k_1,\dotsc,k_{i-1}, k_i-1, k_{i+1},\dotsc, k_{j-1}, k_j+1, k_{j+1} ,\dotsc, k_l} \otimes x\\
&=\z_{j,i}^{-1}\cdot (\zz_{i,1}^{k_1}\dotsm\zz_{i,i-1}^{k_{i-1}})\cdot (\z_{j,1}^{k_1}\dotsm\z_{j,j-1}^{k_{j-1}})\cdot\\
&\quad\quad\quad\quad\quad\quad\quad\quad\quad\quad\quad\quad\quad\quad\quad\quad
e_{k_1,\dotsc,k_{i-1}, k_i-1, k_{i+1},\dotsc, k_{j-1}, k_j+1, k_{j+1} ,\dotsc, k_l} \otimes x.
\end{align*}
Hence $S_i^\ast S_j e_{k_1,\dotsc,k_l}\otimes x =\zz_{ij}S_jS_i^\ast e_{k_1,\dotsc,k_l}\otimes x$, as required.
\end{proof}

We turn to the relations among the $U_i$.

\begin{lemma}\label{res:unitaries_satisfy_relations}
For the unitary operators $U_{l+1},\dotsc,U_n$ on $\ell^2(\NN_0^l)\otimes W$ as defined in \cref{eq:definition_unitary_action} in terms of the unitary operators $\widetilde{U}_{l+1},\dotsc,\widetilde{U}_n$ on $W$ satisfying \cref{eq:classification_unitaries_relation}, we have
\begin{equation*}
U_i^\ast U_j= \zz_{ij}U_j^\ast U_i
\end{equation*}
for all $i,j=l+1,\dotsc,n$ such that $i\neq j$.
\end{lemma}

\begin{proof}
This is immediate from \cref{eq:definition_unitary_action,eq:adjoint_of_unitary_from_definition,eq:classification_unitaries_relation}.
\end{proof}

It remains to consider the relations between the pure isometries and the unitaries.

\begin{lemma}\label{res:isometries_and_unitaries_satisfy_relations}
For the pure isometries $S_1,\dotsc,S_l$ and the unitaries $U_{l+1},\dotsc,U_n$ on $\ell^2(\NN_0^l)\otimes W$ as defined in \cref{eq:definition_pure_isometry_action} and \cref{eq:definition_unitary_action}, respectively, we have
\[
S_i^\ast U_j =\zz_{ij}U_jS_i^\ast
\]
for all $i=1,\dotsc,l$ and $j=l+1,\dotsc,n$.
\end{lemma}

\begin{proof}
We start by establishing that $S_i^\ast U_j=\zz_{ij} U_j S_i^\ast$.

If $k_i=0$, then $S_i^\ast U_j (e_{k_1,\dotsc,k_l}\otimes x)$ and $U_jS_i^\ast (e_{k_1,\dotsc,k_l}\otimes x)$ are both zero. Hence we may suppose that $k_i\geq 1$.
In that case,
\begin{align*}
S_i^\ast U_j (e_{k_1,\dotsc,k_l}\otimes x)  &= (\z_{j,1}^{k_1}\dotsm\z_{j,l}^{k_{l}})\cdot S_i^\ast (e_{k_1,\dotsc,k_l}\otimes \widetilde{U}_j x) \\
&= (\z_{j,1}^{k_1}\dotsm\z_{j,l}^{k_{l}})\cdot (\zz_{i,1}^{k_1}\dotsm\zz_{i,i-1}^{k_{i-1}}) \cdot e_{k_1, \dotsc, k_{i-1}, k_i-1, k_{i+1},\dotsc,k_l} \otimes \widetilde{U}_j x,
\end{align*}
and
\begin{align*}
U_j&S_i^\ast (e_{k_1,\dotsc,k_l}\otimes x)\\
&= (\zz_{i,1}^{k_1}\dotsm\zz_{i,i-1}^{k_{i-1}})\cdot U_j (e_{k_1,\dotsc, k_{i-1}, k_i-1, k_{i+1},\dotsc, k_l} \otimes  x)\\
&= (\zz_{i,1}^{k_1}\dotsm \zz_{i,i-1}^{k_{i-1}})\cdot (\z_{j,1}^{k_1}\dotsm \z_{j,i-1}^{k_{i-1}}{\mathbf{ \z_{j,i}^{k_i-1}}} \z_{j,i+1}^{k_{i+1}}\dotsm\z_{j,l}^{k_{l}})\cdot \\
&\quad\quad\quad\quad\quad\quad\quad\quad\quad\quad\quad\quad\quad\quad\quad\quad\quad\quad\quad\quad\quad e_{k_1,\dotsc,k_{i-1}, k_i-1, k_{i+1} ,\dotsc, k_l} \otimes  \widetilde{U}_j x\\
&= (\zz_{i,1}^{k_1} \dotsm \zz_{i,i-1}^{k_{i-1}})\cdot {\mathbf{\z_{j,i}^{-1}}}\cdot(\z_{j,1}^{k_1} \dotsm \z_{j,i-1}^{k_{i-1}} {\mathbf{\z_{j,i}^{k_i}}}\z_{j,i+1}^{k_{i+1}}\dotsm\z_{j,l}^{k_{l}})  \cdot \\
&\quad\quad\quad\quad\quad\quad\quad\quad\quad\quad\quad\quad\quad\quad\quad\quad\quad\quad\quad\quad\quad e_{k_1,\dotsc,k_{i-1}, k_i-1, k_{i+1} ,\dotsc, k_l} \otimes  \widetilde{U}_j x\\
&={\z_{j,i}^{-1}}\cdot (\zz_{i,1}^{k_1}\dotsm\zz_{i,i-1}^{k_{i-1}}) \cdot (\z_{j,1}^{k_1}\dotsm\z_{j,l}^{k_{l}})\cdot e_{k_1,\dotsc,k_{i-1}, k_i-1, k_{i+1} ,\dotsc, k_l} \otimes  \widetilde{U}_j x.
\end{align*}
Therefore, $S_i^\ast U_j= \zz_{ij} U_jS_i^\ast$, as required. Taking the adjoint of this relation and using that $\z_{ij}=\zz_{ji}$ shows that $U_j^\ast S_i=\zz_{ji} S_i U_j^\ast$.
\end{proof}

We have now completed the verification that $(S_1,\dotsc,S_l,U_{l+1},\dotsc,U_n)$ is an $n$-tuple of \dnci\ on $\ell^2(\NN_0^l)\otimes W$. It is an easy consequence of \cref{eq:adjoint_of_pure_isometry} and the unitarity of the $\widetilde{U}_i$ that the $\{1,\dotsc,l\}$-wandering subspace of $(S_1,\dotsc,S_l,U_{l+1},\dotsc,U_n)$ is $e_{0,\dotsc,0}\otimes W$, thus explaining the choice of the letter. We shall identify this space with $W$. With this identification, the $\{1,\dotsc,l\}$-wandering data  $\wdata_{\{1,\dotsc,l\}}$ of $(S_1,\dotsc,S_l,U_{l+1},\dotsc,U_n)$ are $(\idop_W,\widetilde{U}_{l+1},\ldots,\widetilde{U}_n)$. As a consequence of \cref{res:pure_isometry,res:unitary,res:space_decomposition_types_of_actions}, all other wandering data  are zero tuples.

For $l=1,\dotsc,n-1$, we shall refer to the $n$-tuple $(S_1,\dotsc,S_l,U_{l+1},\dotsc,U_n)$, where the pure isometries $S_1,\dotsc,S_l$ and the unitary operators $U_{l+1},\dotsc,U_n$ on $\ell^2(\NN_0^l)\otimes W$ are as defined in \cref{eq:definition_pure_isometry_action} and \cref{eq:definition_unitary_action}, respectively, and where the unitary operators $\widetilde{U}_{l+1},\dotsc,\widetilde{U}_n$ on $W$ satisfy \cref{eq:classification_unitaries_relation}, as the \emph{standard $n$-tuple with $\{1,\dotsc,l\}$-wandering data  $(\idop_W,\widetilde{U}_{l+1},\dotsc,\widetilde{U}_n)$}.

It remains to consider the case where $l=n$. In that case,  \cref{res:space_decomposition_types_of_actions,res:space_decomposition_for_given_types_of_actions} show again that there exists a subspace $W$ of $H$ such that
\[
H=\bigoplus_{k_1,\dotsc,k_n=0}^\infty S_1^{k_n}\dotsm S_n^{k_n} (W).
\]
One then again defines $\varphi$ as in \cref{eq:phi_definition}. In this case, there are no unitary operators to transfer to $\ell^2(\NN^n)\otimes W$, and one is left with only \cref{eq:classification_pure_isometry_action}, where then $l=n$. This is then the Wold decomposition for the $n$-tuple $(S_1,\dotsc,S_n)$. In obvious analogy with the classical result for one pure isometry, the action of the tuple is a Hilbert sum of copies of the case where $W=\CC$.

Turning the tables, one defines, for $i=1,\dotsc,n$, the operators $S_1,\ldots,S_n$ on $\ell^2(\NN^n)\otimes W$ by
\begin{equation}\label{eq:case_l_equal_to_n}
S_i(e_{k_1,\dotsc,k_n}\otimes x)\coloneqq\z_{i,1}^{k_1}\dotsm\z_{i,i-1}^{k_{i-1}}\ e_{k_1, \dotsc, k_{i-1}, k_i+1, k_{i+1} ,\dotsc, k_n} \otimes x
\end{equation}
for all $k_1,\dotsc,k_n\geq 0$ and $x\in W$. Then the proof of \cref{res:isometries_satisfy_relations}, which applies equally well if $l=n$, shows that $(S_1,\dotsc,S_n)$ is an $n$-tuple of \dnci\ on $\ell^2(\NN_0^n)\otimes W$. It is evident from \cref{eq:adjoint_of_pure_isometry} that the $\{1,\dotsc,n\}$-wandering subspace of $(S_1,\dotsc,S_n)$ as in \cref{def:wandering_subspace_and_core} is $e_{0,\dotsc,0}\otimes W$. We shall identify this space with $W$ again. With this identification, the $\{1,\dotsc,n\}$-wandering data  $\wdata_{\{1,\ldots,n\}}$ of $(S_1,\dotsc,S_n)$ reduce to the 1-tuple $(\idop_W)$. As a consequence of \cref{res:pure_isometry,res:space_decomposition_types_of_actions}, all other wandering data  are zero tuples.

We shall call $(S_1,\dotsc,S_n)$ the \emph{standard $n$-tuple with $\{1,\ldots,n\}$-wandering data  $(\idop_W)$}.

It is now time to tie up the obvious loose end in the above:  for $l=0,\dotsc,n-1$, we still need to find unitary operators $\widetilde{U}_{l+1},\dotsc,\widetilde{U}_n$ on a non-zero Hilbert space $W$ that satisfy \cref{eq:classification_unitaries_relation}. Only this will give us non-zero examples of standard $n$-tuples for such $l$. With the results above available, this is now easily done. We simply mimic \cref{eq:definition_pure_isometry_action}, where we now allow also negative integer indices. To be precise: take $W=\ell^2(\ZZ^{n-l})$, and denote the canonical orthonormal basis elements by $e_{k_{l+1},k_{n}}$ for $k_{l+1},\ldots,k_n\in\ZZ$.  For $i=l+1,\ldots,n$, we define

\begin{equation}\label{eq:definition_non-commutative_torus}
\widetilde{U}_i e_{k_{l+1},\dotsc,k_n} \coloneqq\z_{i,l+1}^{k_{l+1}}\dotsm\z_{i,i-1}^{k_{i-1}}\ e_{k_{l+1},\dotsc, k_{i-1}, k_i+1, k_{i+1} ,\dotsc, k_n}
\end{equation}
for $k_{l+1},\ldots,k_n\in\ZZ$. Evidently,
\begin{equation}\label{eq:adjoint_of_unitary_from_non-commutative_torus}
\widetilde{U}_i^\ast e_{k_{l+1},\dotsc,k_n} =\zz_{i,l+1}^{k_{l+1}}\dotsm\zz_{i,i-1}^{k_{i-1}}\ e_{k_{l+1},\dotsc, k_{i-1}, k_i+1, k_{i+1} ,\dotsc, k_n}
\end{equation}
for $k_{l+1},\ldots,k_n\in\ZZ$.

\begin{lemma}\label{res:non-commutative_torus_satisfies_relations}
For the unitary operators ${\widetilde U}_{l+1},\dotsc,{\widetilde U}_n$ on $\ell^2(\ZZ^{n-l})$ as defined in \cref{eq:definition_non-commutative_torus}, we have
\begin{equation*}
\widetilde{U}_i^\ast \widetilde{U}_j^{\phantom{\ast}}\!= \zz_{ij}\widetilde{U}_j^\ast \widetilde{U}_i^{\phantom{\ast}}\!
\end{equation*}
for all $i,j=l+1,\dotsc,n$ such that $i\neq j$.
\end{lemma}

\begin{proof}
This has essentially already been done in the proof of \cref{res:isometries_satisfy_relations}. Comparing that context with the present one, there are presently no cases that need to be considered separately when indices labelling the orthonormal basis are zero. We are only left with the analogue of the computational part of the proof of \cref{res:isometries_satisfy_relations}. For this, we need merely note that this part of the proof of \cref{res:isometries_satisfy_relations} does not use that the indices $k_1,\dotsc,k_l$ labelling the elements of the orthonormal basis are non-negative. It is sufficient to have \cref{eq:definition_pure_isometry_action} and the first line of \cref{eq:adjoint_of_pure_isometry} for all indices under consideration.
Since \cref{eq:definition_non-commutative_torus} and \cref{eq:adjoint_of_unitary_from_non-commutative_torus} have a structure that is completely analogous to that of \cref{eq:definition_pure_isometry_action} and the first line of \cref{eq:adjoint_of_pure_isometry}, respectively, a completely analogous computation establishes the relations in the present lemma. As earlier, it originates from the fact that for a pair of different indices labelling the elements of the orthonormal basis there is always one that precedes the other.
\end{proof}

We have now described all $n$-tuples \Vtuple\ of \dnci\ that are of `pure type' up to unitary equivalence, and we summarise this description in the next result. We emphasise that \cref{res:non-commutative_torus_satisfies_relations} (which is not visible in the statement) is necessary to show that it has substance. The convention in its formulation is that lists where the lower bound of the index exceeds the upper bound are absent.

\begin{theorem}\label{res:unitary_equivalence_to_standard_realisation}
Let $l$ be such that $0\leq l\leq n$.
\begin{enumerate}
\item Suppose that $\widetilde{U}_{l+1},\dotsc,\widetilde{U}_n$ are unitary operators on a Hilbert space $W$ satisfying \cref{eq:classification_unitaries_relation}. Then the standard $n$-tuple $(S_1,\dotsc,S_l,U_{l+1},\dotsc,U_n)$ with $\{i: 1\leq i\leq l\}$-wandering data  $(\idop_W,\widetilde{U}_{l+1},\dotsc,\widetilde{U}_n)$ is an $n$-tuple of \dnci\ on $\ell^2(\NN_0^l)\otimes W$. The associated wandering subspace $W_{\{i: 1\leq i\leq l\}}$ of  $(S_1,\dotsc,S_l,U_{l+1},\dotsc,U_n)$ is canonically isomorphic to $W$, and the $\{i: 1\leq i\leq l\}$-wandering data  are then $(\idop_W,\widetilde{U}_{l+1},\dotsc,\widetilde{U}_n)$. All other wandering data  are zero tuples. The operators $S_1,\dotsc,S_{l}$ are pure isometries, and the operators $U_{l+1},\dotsc,U_n$ are unitary.
\item Suppose that $(V_1,\dotsc, V_n)$ is an $n$-tuple of \dnci, that the first $l$ of these are pure isometries, and that the final $(n-l)$ ones are unitary. Let $W$ be the wandering subspace of $(V_1,\dotsc, V_n)$ as in \cref{res:space_decomposition_for_given_types_of_actions}, and let $\widetilde{U}_{l+1},\dotsc,\widetilde{U}_n$ denote the unitary restrictions of the respective operators $U_{l+1},\ldots,U_n$ to $W$.
Then the $n$-tuple $(V_1,\ldots,V_n)$ is unitarily equivalent to the standard $n$-tuple $(S_1,\dotsc,S_l,U_{l+1},\dotsc,U_n)$ with $\{i: 1\leq i\leq l\}$-wandering data  $(\idop_W,\widetilde{U}_{l+1},\dotsc,\widetilde{U}_n)$
\end{enumerate}
\end{theorem}

Naturally, one can carry out the construction of standard $n$-tuples for an arbitrary subset $A\subseteq\{1,\dotsc,n\}$ that is not necessarily an initial segment, and find an $n$-tuple \Vtuple\ of \dnci\ such that the $\V_i$ for $i\in A$ are pure isometries and the $\V_i$ for $i\in A^\upc$ are unitary with pre-given restrictions to the pre-given wandering subspace of \Vtuple. One needs to be careful, though,  when defining the coupling between the list of restrictions in the $A$-wandering data  and the corresponding unitaries in the newly constructed $n$-tuple \Vtuple, because the structure constants of these restrictions are inherited by the corresponding unitaries in \Vtuple; see \cref{res:unitaries_satisfy_relations}. We need to make sure that these inherited constants are the corresponding $\z_{ij}$ in \cref{eq:relations_for_operators}.   Therefore, if $|A|=l$ and $A^\upc=\{j_1,\dotsc,j_{n-l}\}$ with $j_1<\dotsb j_{n-l}$, we insist that the pre-given $A$-wandering data  in $\wdata_{A}$ are listed as $(\idop_W,{\widetilde{U}}_{j_1},\dotsc,{\widetilde{U}}_{j_{n-l}})$, and are such that ${\widetilde{U}}_{j_r}^\ast {\widetilde{U}}_{j_s}^{\phantom{\ast}}\,=\zz_{j_rj_s} {\widetilde{U}}_{j_s}^{\phantom{\ast}}\,{\widetilde{U}}_{j_r}^\ast$ for all $r,s=1,\dotsc,n-l$ such that $r\neq s$. This requirement is the counterpart of the ordering of the indices as required in the definition of the $A$-wandering data  of an already existing $n$-tuple \Vtuple\ of \dnci; see \cref{def:wandering_subspace_and_core}.

With this requirement on the ordering of the indices in place on two occasions, the obvious analogue of \cref{res:unitary_equivalence_to_standard_realisation} holds.
Firstly, if $|A|=l$ and $A^\upc=\{j_1,\dotsc,j_{n-l}\}$ with $j_1<\dotsb j_{n-l}$, then one can construct an $n$-tuple \Vtuple\ of \dnci\ that has pre-given $A$-wandering data  $\wdata_A=(\idop_W,{\widetilde{U}}_{j_1},\dotsc,{\widetilde{U}}_{j_{n-l}})$ while all other wandering data  are zero tuples, and where the $\V_i$ for $i\in A$ are pure isometries and the $\V_i$ for $i\in A^\upc$ are unitary. This is called the \emph{standard $n$-tuple of \dnci\ with $A$-wandering data  $(\idop_W,{\widetilde{U}}_{j_1},\dotsc,{\widetilde{U}}_{j_{n-l}})$}, and it is denoted by $\standard_{\wdata_A}$. Secondly, if \Vtuple\ is an $n$-tuple of \dnci\ such that the $\V_i$ for $i\in A$ are pure isometries and the $\V_i$ for $i\in A^\upc$ are unitary, and if its $A$-wandering data $\wdata_A$ are as in \cref{def:wandering_subspace_and_core}, then \Vtuple\ is unitarily equivalent to ${\standard}_{{\wdata}_A}$.

Now that we have described all $n$-tuples of `pure type' (essentially on basis of \cref{res:space_decomposition_for_given_types_of_actions}), the following is clear from \cref{res:space_decomposition_types_of_actions}. If $\z_{ij}=1$ for all $i$ and $j$ one retrieves a particular case of \cite[Theorem 3.1]{an_huef_raeburn_tolich:2015}.

\begin{theorem}[Wold decomposition]\label{res:wold_decomposition} Let $\Vtuple$ be an $n$-tuple of \dnci. Then $\Vtuple$ is unitarily equivalent to the Hilbert direct sum $\bigoplus_{A\subseteq\{1,\dotsc,n\}} {\standard}_{{\wdata}_A}$, where $\wdata_A$ denotes the $A$-wandering data  of \Vtuple.
\end{theorem}

There does not seem much more that we can add concerning the structure of $n$-tuples of \dnci. If $n=1$, then the classical Wold decomposition for an isometry provides no further information about the unitary component, and here we have something similar for the (this role has now become obvious) representations of non-commutative tori on the wandering subspaces.

We shall take up the parameterization of the unitary equivalence classes of $n$-tuples in Section~\ref{sec:classification}.

\begin{remark}\label{rem:part_II}
The \cref{eq:definition_pure_isometry_action,eq:definition_unitary_action,eq:definition_non-commutative_torus} provide a more or less elementary example of $n$ \dnci\ on $\ell^2(\NN_0^l)\otimes\ell^2(\ZZ^{n-l})$ such that the first $l$ are pure isometries and the final $(n-l)$ are unitary. It can actually be shown that the \Calgebra\ that is generated by these operators is the universal \Calgebra\ that is generated by $n$ \dnci\ (still with the same structure constants) such that the final $(n-l)$ of these are unitary. In particular, for $l=0$, one retrieves the known fact (see \cite{rieffel:1990}) that the \Calgebra\ that is generated by unitary operators ${\widetilde U},\dotsc,\widetilde{U}_n$ on $\ell^2(\ZZ^n)$ as in \cref{eq:definition_non-commutative_torus} is isomorphic to the non-commutative $n$-torus. At the other extreme, if $l=n$, then one sees that the Fock representation of the universal \Calgebra\ that is generated by $n$ doubly non-commuting isometries is faithful. This is already known, but to conclude this from the existing literature one has to distinguish two cases, and combine \cite[Proposition~8]{proskurin_2000} and \cite[Corollary~1]{kabluchko:2001}. Our proof is uniform.

We shall report on these universal \Calgebras, their interrelations, and their representations in a separate paper.
\end{remark}

\section{Classification}\label{sec:classification}

It is now easy to classify $n$-tuples of \dnci\ up to unitary equivalence. \cref{res:wold_decomposition} provides an obvious candidate for a classifying invariant, namely, the collections of unitary equivalence classes of representations of the non-commutative tori that are naturally associated with the subsets of $\{1,\ldots.n\}$, borrowing their structure constant from those for \Vtuple. This is indeed the case.

First, however, we include the following result.  Though completely elementary, the observation should still be made. We deviate for once from our convention that the structure constants $\z_{ij}$ are fixed.

\begin{lemma}\label{res:equivalent_tuple_must_have_same_constants}
Let \Vtuple\ be an $n$-tuple of \dnci\ with structure constants $\z_{ij}$ on a non-zero Hilbert space $H$, and let $(\V_1^\prime,\dotsc,\V_{n}^\prime)$ be an $n$-tuple of \dnci\ with structure constants $\z_{ij}^\prime$ on a non-zero Hilbert space $H^\prime$. If $(\V_1,\dotsc,\V_{n})$ and $(\V_1^\prime,\dotsc,\V_{n}^\prime)$ are unitarily equivalent, then $\z_{ij}=\z_{ij}^\prime$ for all $i\neq j$.
\end{lemma}

\begin{proof}
It follows from $\V_i^\prime \V_j^\prime=\z_{ij}^\prime \V_j^\prime \V_i^\prime$ and the existence of an equivalence that $\V_i \V_j=\z_{ij}^\prime \V_j \V_i$. Hence $\z_{ij}^\prime \V_j \V_i=\z_{ij}\V_j \V_i$. Since $H$ is not the zero space and the isometries are injective, we must have $\z_{ij}=\z_{ij}^\prime$.
\end{proof}

Hence we can safely return to our convention that the $\z_{ij}$ are fixed and that they are suppressed in the notation. For such fixed structure constants, a part of the classification up to unitary equivalence is provided by the following.

\begin{theorem}\label{res:unitary_equivalence}
Let $(\V_1,\dotsc,\V_{n})$ be an $n$-tuple of \dnci\ on a Hilbert space $H$, and let $(\V_1^\prime,\dotsc,\V_{n}^\prime)$ be an $n$-tuple of \dnci\ on a Hilbert space $H^\prime$, with identical structure constants in the relations. Then the following are equivalent:
\begin{enumerate}
\item\label{part:classification_tuples} The $n$-tuples $(\V_1,\dotsc,\V_{n})$ and $(\V_1^\prime,\dotsc,\V_{n}^\prime)$ are unitarily equivalent;
\item\label{part:classification_cores} The $A$-wandering data  $\wdata_A$ of $(\V_1,\dotsc,\V_{n})$ and the $A$-wandering data  $\wdata_A^\prime$ of $(\V_1^\prime,\dotsc,\V_{n}^\prime)$ are unitarily equivalent for all $A\subseteq\{1,\dotsc,n\}$.
\end{enumerate}
\end{theorem}

\begin{proof}
We prove that \partref{part:classification_tuples} implies \partref{part:classification_cores}. If $\varphi: H\mapsto H^\prime$ is a unitary equivalence of $(\V_1,\dotsc,\V_{n})$ and $(\V_1^\prime,\dotsc,\V_{n}^\prime)$, then the definition of the wandering subspaces in \cref{res:space_decomposition_for_given_types_of_actions} shows that $\varphi(W_A)=W^\prime_A$ for all sets $A$ of indices. Alternatively, an inspection of the proof of \cref{res:space_decomposition_for_given_types_of_actions}  shows that $P_{W_A}$ and $P_{W^\prime_A}$ are strong operator limits of polynomials in the isometries and their adjoints in the respective spaces, so that $\varphi$ is also a unitary equivalence between these projections. At any rate, it is thus clear that $\varphi$ also implements a unitary equivalence between all respective $A$-wandering data.

We prove that \partref{part:classification_cores} implies \partref{part:classification_tuples}. In view of \cref{res:space_decomposition_types_of_actions}, it is sufficient to show that, for all $A\subseteq\{1,\dotsc,n\}$, the restriction of $(\V_1,\dotsc,\V_{n})$ to $H_A$ is unitarily equivalent to the restriction of $(\V_1^\prime,\dotsc,\V_{n}^\prime)$ to $H_A^\prime$. Fix $A=\{i_1,\dotsc,i_l\}$ with $i_1<\dotsb<i_l$, and let $\varphi_{W_A}:W_A\mapsto W_A^\prime$ be an isomorphism that is a unitary equivalence between $\wdata_A$ and $\wdata_A^\prime$. We suppose that $1\leq l\leq n-1$; the cases where $l=0$ and $l=n$ are handled similarly. Using \cref{res:space_decomposition_for_given_types_of_actions} for both $H_A$ and $H_A^\prime$, we can define a map $\varphi_A:H_A\mapsto H_A^\prime$ by
\[
\varphi\left(\sum_{k_1,\dotsc,k_l=0}^\infty \V_{i_1}^{k_{i_1}}\dotsm \V_{i_l}^{k_{i_l}} x_{k_{i_1},\dotsc,k_{i_l}}\right) = \sum_{k_1,\dotsc,k_l=0}^\infty \V_{i_1}^{\prime k_{i_1}}\dotsm \V_{i_l}^{\prime k_{i_l}} \varphi(x_{k_{i_1},\dotsc,k_{i_l}})
\]

where $x_{k_{i_1},\dotsc,k_{i_l}}\in W_A$. On noting that
\[
\norm{\V_{i_1}^{k_{i_1}}\dotsm \V_{i_l}^{k_{i_l}} x_{k_{i_1},\dotsc,k_{i_l}}}=\norm{x_{k_{i_1},\dotsc,k_{i_l}}}=\norm{\V_{i_1}^{\prime k_{i_1}}\dotsm \V_{i_l}^{\prime k_{i_l}} \varphi(x_{k_{i_1},\dotsc,k_{i_l}}}
\]
it becomes clear that $\varphi_A$ is an isomorphism between $H_A$ and $H_A^\prime$. It is then easy to see that $\varphi_A$ is a unitary equivalence between the restriction of $(\V_1,\dotsc,\V_{n})$ to $H_A$ and of $(\V_1^\prime,\dotsc,\V_{n}^\prime)$ to $H_A^\prime$. The reason is that, for $i\in A$, one picks up the same constant when moving $\V_i$ to its proper position in a product $\V_{i_r}\cdot \V_{i_1}^{k_{i_1}}\dotsm \V_{i_l}^{k_{i_l}}$ as is picked up when doing the same sorting for $\V_{i}^\prime\cdot \V_{i_1}^{\prime k_{i_1}}\dotsm \V_{i_l}^{\prime k_{i_l}}$, and that, for $i\in A^{\upc}$, one picks up the same constants when moving the factors $\V_{i}$ and $\V_{i}^\prime$ through the entire product.
\end{proof}

For every $A\subseteq\{1,\dotsc,n\}$, let $\universal_A$ be the non-commutative $(n-|A|)$-torus that is determined by the indices in $A^\upc$ and the corresponding structure constants $z_{ij}$ for $i,j\in A^\upc$ with $i\neq j$. It is the universal \Calgebra\ that is generated by unitaries satisfying the pertinent relations; for $A=\{1,\dotsc,n\}$ it equals $\CC$. If $A$ is a set of indices, then the possible $A$-wandering data $\wdata_A$ that are the input for the construction of the corresponding standard $n$-tuple $\standard_{\wdata_A}$ are obviously in bijection with the unital representations of $\universal_A$, and two such representations of $\universal_A$ are unitarily equivalent if and only if the corresponding $A$-wandering data are unitarily equivalent.

\begin{theorem}[Classification]\label{res:classification}
The unitary equivalence classes of $n$-tuples of \dnci\ are in natural bijection with enumerations of $2^n$ unitary equivalence classes of unital representations of the non-commutative tori ${\universal}_A$, as $A$ ranges over the power set of $\{1,\dotsc,n\}$.  The bijection is obtained by listing, for a given unitary equivalence class with representing $n$-tuple $\Vtuple$, the unitary equivalences classes of the representations of $\universal_A$ that are associated with the $A$-wandering data $\wdata_A$ of \Vtuple, as $A$ ranges over the power set of $\{1,\dotsc,n\}$.
\end{theorem}

\begin{proof}
\cref{res:unitary_equivalence} shows that the map as described is well-defined and injective. The existence of the standard $n$-tuples  $\standard_{\wdata_A}$ (as discussed following \cref{res:unitary_equivalence_to_standard_realisation}) and then that of their Hilbert direct sums shows that it is surjective.
\end{proof}

As an application, we include a description and parameterization of the irreducible $n$-tuples in this section. This improved version (there is now a classification part) of  \cite[Theorem~2]{proskurin_2000} is now easily seen to be a consequence of the results for arbitrary $n$-tuples.

\begin{theorem}\label{res:irreducibility} Let \Vtuple\ be an $n$-tuple of \dnci\ on a Hilbert space $H$. Then the following are equivalent:
\begin{enumerate}
\item\label{part:tuple_acts_irreducibly} $H$ has only trivial subspaces that are invariant under the \Calgebra\ that is generated by the operators in the $n$-tuple \Vtuple;
\item\label{part:torus_acts_irreducibly} There exists a \uppars{possibly empty} set of indices $A\subseteq\{1,\dotsc,n\}$ such that \Vtuple\ is unitarily equivalent to a standard $n$-tuple $\standard_A$ with the property that the $A$-wandering subspace $W_A$ of the pertinent Hilbert space has only trivial subspaces that are invariant under the \Calgebra\ of operators on $W_A$ that is generated by the operators in the $A$-wandering data $\wdata_A$ of $\standard_A$.
\end{enumerate}
In that case, if $H\neq\{0\}$, two such $n$-tuples are unitarily equivalent if and only if the corresponding sets of indices in part~\textup{\partref{part:torus_acts_irreducibly}} are equal and the wandering data for these equal sets of indices are unitarily equivalent.
\end{theorem}

\begin{proof}
We prove that \partref{part:tuple_acts_irreducibly} implies \partref{part:torus_acts_irreducibly}. If $H=\{0\}$, then one can take a zero standard $n$-tuple on a zero space. If $H\neq\{0\}$, then
\cref{res:space_decomposition_types_of_actions} shows that precisely one of the spaces $H_A$ is non-zero. According to (the general analogue of) \cref{res:unitary_equivalence_to_standard_realisation}, \Vtuple\ is then unitarily equivalent to a standard $n$-tuple $\standard_A$. It is then immediate from the structure of the standard $n$-tuples that the irreducibility of the action of the $n$-tuple on $H_A$ implies that the action of the operators in the wandering data on $W_A$ must likewise be irreducible.

We prove that \partref{part:torus_acts_irreducibly} implies~\partref{part:tuple_acts_irreducibly}. If $W_A=\{0\}$, then $H=\{0\}$, and we are done. Hence we suppose that $W_A\neq\{0\}$. Resorting to the earlier notation, we suppose for simplicity that the first $l$ operators in the $n$-tuple are pure isometries, and that the remaining ones are unitary.
If $l=0$, then there is nothing to prove. Hence we suppose that $l\geq 1$. Suppose that $L\subseteq\ell^2(\NN_0^{l})\otimes W_A$ is a non-zero subspace that reduces the operators in the associated standard $n$-tuple. Choose a non-zero $x\in L$. Then
\(
x=\sum_{k_1,\dotsc,k_l\geq 0} e_{k_1,\dotsc,k_l}\otimes x_{k_1,\dotsc,k_l}
\)
with $x_{k_1,\dotsc,k_l}\in W_A$ for all $k_1,\dotsc,k_l\geq 0$. Let $m_l$ be the minimal non-negative integer such that there exist $k_1,\ldots,k_{l-1}\geq 0$ with $x_{k_1,\dotsc,k_{l-1},m_l}\neq 0$. Then
\[
x-S_l^{m_l+1}S_l^{\ast(m_l+1)}x=\sum_{k_1,\dotsc,k_{l-1}\geq 0} e_{k_1,\dotsc,k_{l-1},m_l}\otimes x_{k_1,\dotsc,k_{l-1},m_l}
\]
is a non-zero element of $L$. After applying $S^{\ast m_l }$ to this element, subsuming the resulting unimodular constants into the $x_{k_1,\dotsc,k_{l-1},m_l}$, and relabelling the latter, we see that $L$ has a non-zero element
\(
x^\prime=\sum_{k_1,\dotsc,k_l\geq 0} e_{k_1,\dotsc,k_{l-1},0}\otimes x_{k_1,\dotsc,k_{l-1},0}^\prime.
\)
Repeating this procedure $(l-1)$ times, we see that $L$ contains a non-zero element $e_{0,\dotsc,0}\otimes \widetilde x$. We can now let the \Calgebra\ that is generated by $U_{l+1},\dotsc,U_n$ act on $e_{0,\dotsc,0}\otimes \widetilde x$, and the irreducibility of $W_A$ then shows that $L$ contains $e_{0,\dotsc,0}\otimes W_A$. The action of the \Calgebra\ that is generated by $S_1,\ldots,S_l$ then yields that $L=\ell^2(\NN_0^{l})\otimes W_A$.

The proof of the equivalence of the parts~\partref{part:tuple_acts_irreducibly} and \partref{part:torus_acts_irreducibly} is now complete.

For the remaining statement, we note that \cref{res:space_decomposition_types_of_actions} implies that the corresponding sets of indices are equal. Then the unitary equivalence of the wandering data follows from \cref{res:unitary_equivalence}.

\end{proof}

The following is now clear.

\begin{corollary}
The unitary equivalence classes of the non-zero irreducible representations of the \Calgebra\ that is generated by $n$ isometries satisfying \cref{eq:relations_for_operators} are parameterized by the unitary equivalence classes of the non-zero irreducible representations of the $2^n$ non-commutative tori that are naturally associated with the structure constants $\z_{ij}$ in \cref{eq:relations_for_operators}.
\end{corollary}

\begin{example}\label{ex:wold_decomposition_for_one_isometry}

We shall now discuss how the results work out if $n=1$, when there is only one isometry $\V$.

\cref{res:wold_decomposition} yields that the 1-tuple $(\V)$ is unitarily equivalent to ${\standard}_{\emptyset}\oplus{\standard}_{\{1\}}$, where ${\standard}_{\emptyset}$ is the standard 1-tuple with $\emptyset$-wandering data  $(\idop_{W_{\emptyset}},\V|_{W_{\emptyset}})$, and where ${\standard}_{\{1\}}$ is the standard 1-tuple with $\{1\}$-wandering data  $(\idop_{W_{\{1\}}})$. According to (the discussion preceding) \cref{res:classification}, the $\emptyset$-wandering data $(\idop_{W_{\emptyset}},\V|_{W_{\emptyset}})$ arise from a unital representation of the non-commutative 1-torus. This is equivalent to saying that $\V|_{W_{\emptyset}}$ is a unitary operator on some (possibly zero) Hilbert space $W_{\emptyset}$. The structure of the standard 1-tuple ${\standard}_{\emptyset}$ is simply that of the unitary operator $\V|_{W_{\emptyset}}$ acting on $W_{\emptyset}$; this is the case where $l=0$ that is considered first in the beginning of Section~\ref{sec:Wold_decomposition_and_examples}.

According to (the discussion preceding) \cref{res:classification}, the $\{1\}$-wandering data $(\idop_{W_{\{1\}}})$ arise from a unital representation of the non-commutative 0-torus. This is equivalent to saying that $W_{\{1\}}$ is an arbitrary (possibly zero) Hilbert space. According to \cref{eq:case_l_equal_to_n} (we are in the case where $l=n$), the structure of the standard 1-tuple ${\standard}_{\{1\}}$ is then that of $S\otimes\idop_{W_{\{1\}}}$ acting on $\ell^2(\NN_0)\otimes W_{\{1\}}$; here $S$ is the unilateral shift.

Thus the classical Wold decomposition is retrieved: $\V$ is unitarily equivalent to the Hilbert direct sum of a unitary operator and copies of the unilateral shift.

According to \cref{res:unitary_equivalence}, the unitary equivalence classes of 1-tuples $(\V)$ of isometries are in bijection with the enumerations of both a unitary equivalence class of $(\idop_{W_{\emptyset}},\V|_{W_{\emptyset}})$ and a unitary equivalence class of $(\idop_{W_{\{1\}}})$. That is: an isometry is determined, up to unitary equivalence, by the unitary equivalence class of its unitary part and the multiplicity of the unilateral shift.

Although there is no mathematical gain for $n=1$ by retrieving a decomposition that was used as a key ingredient for the general result to begin with, and a classification result that could have been cited from the literature, it still seems satisfactory to see how the known classifying invariants for isometries fit into a more general picture, in which these $2^1$ invariants are the unitary equivalence classes of representations of the non-commutative $0$- and $1$-tori.
\end{example}

\section{Dilation theorem}\label{sec:dilation}

It is now easy to prove a dilation theorem. Just as in the case where $n=1$, now that a Wold decomposition is available from Section~\ref{sec:Wold_decomposition_and_examples}, this is merely a matter of allowing negative indices where needed.

We start with the cases where there are only pure isometries and unitaries. As earlier, the convention in its formulation is that lists where the lower bound of the index exceeds the upper bound are absent.

\begin{proposition}\label{res:dilation_for_fixed_A}
Let $l\geq 0$, and suppose that $(S_1,\dotsc,S_l,U_{l+1},\dotsc,U_{n})$ is an $n$-tuple of \dnci\ on a Hilbert space $H$, where $S_1,\dotsc,S_n$ are pure isometries and $U_{l+1},\dotsc,U_n$ are unitary.

Then there exists a Hilbert space $\mathcal K$ containing $H$, with projection $P_H:\mathcal K\to H$, and unitary operators $\Ue_1,\dotsc, \Ue_l,\Ue_{l+1},\dotsc,\Ue_n$ on $\mathcal K$ such that
\begin{enumerate}
\item\label{part:dilated_tuple_is_doubly_non-commutative} $(\Ue_1,\dotsc, \Ue_n)$ is an $n$-tuple of \dnci;
\item\label{part:invariance} $\Ue_1,\dotsc,\Ue_n$ leave $H$ invariant;
\item\label{part:restrictions} The restriction of $\Ue_1,\dotsc, \Ue_n$ to $H$ is $S_1,\dotsc,S_l$, $U_{l+1},\dotsc,U_{n}$, respectively;
\item\label{part:invariance_under_adjoints} $\Ue_{l+1}^\ast,\dotsc,\Ue_n^\ast$ leave $H$ invariant;
\item\label{part:restrictions_of_adjoints} The restriction of $\Ue_{l+1}^\ast,\dotsc,\Ue_n^\ast$ to $H$ is $U_{l+1}^\ast,\dotsc,U_{n}^\ast$, respectively;
\item\label{part:compressions} $S_i^\ast=(P_H \circ \Ue_i^\ast)|_H$ for $i=1,\dotsc,l$. 
\end{enumerate}
\end{proposition}

In view of the parts~\partref{part:invariance_under_adjoints} and~\partref{part:restrictions_of_adjoints}, one can add $U_i^\ast=(P_H \circ \Ue_i^\ast)|_H$ to the list in part~\partref{part:compressions} for the remaining indices $i=l+1,\dotsc,n$, which is convenient for the proof of \cref{res:dilation} below.

\begin{proof}
If $l=0$ we can take $\mathcal K=H$. Hence we suppose that $l\geq 1$.
In view of \cref{res:unitary_equivalence_to_standard_realisation} we may suppose that $(S_1,\dotsc,S_l,U_{l+1},U_{n})$ is the standard $n$-tuple of \dnci\ with $\{1,\dotsc,l\}$-wandering data  $(\idop_W,{\widetilde U}_{l+1},\dotsc,{\widetilde U}_{n})$, where the $\widetilde{U}_i$ (if any) are the restrictions of the respective $U_i$ to the wandering subspace $W$ of $(S_1,\dotsc,S_l,U_{l+1},U_{n})$.  In that model, we need merely extend the range of the indices labelling the elements of the orthonormal basis of $\ell^2(\NN_0^l)$ in \cref{eq:definition_pure_isometry_action,eq:definition_unitary_action} if $l\leq n-1$, or the range of the indices labelling the elements of the orthonormal basis of $\ell^2(\NN^n)$ in \cref{eq:case_l_equal_to_n} if $l=n$. We make this explicit.

If $l\leq n-1$, then, for $i=1,\dotsc,l$, we define
\begin{equation*}
\Ue_i(e_{k_1,\dotsc,k_l}\otimes x)\coloneqq\z_{i,1}^{k_1}\dotsm\z_{i,i-1}^{k_{i-1}}\ e_{k_1, \dotsc, k_{i-1}, k_i+1, k_{i+1} ,\dotsc, k_l} \otimes x
\end{equation*}
for all $k_1,\dotsc,k_l\in\ZZ$ and $x\in W$, and, for $i=l+1,\dotsc,n$, we define
\begin{equation*}
\Ue_i(e_{k_1,\dotsc,k_l}\otimes x) \coloneqq \z_{i,1}^{k_1}\dotsm\z_{i,l}^{k_{l}}\ e_{k_1,\dotsc,k_l}\otimes \widetilde{U}_i x
\end{equation*}
for all $k_1,\dotsc,k_l\in\ZZ$ and $x\in W$.

If $l=n$, then, for $i=1,\dotsc,l$, we define
\begin{equation*}
\Ue_i(e_{k_1,\dotsc,k_n}\otimes x)\coloneqq\z_{i,1}^{k_1}\dotsm\z_{i,i-1}^{k_{i-1}}\ e_{k_1, \dotsc, k_{i-1}, k_i+1, k_{i+1} ,\dotsc, k_n} \otimes x
\end{equation*}
for all $k_1,\dotsc,k_n\in\ZZ$ and $x\in W$.

For all $l\leq n$, $\Ue_1,\dotsc,\Ue_n$ are clearly unitary operators on $\ell^2(\ZZ^l)\otimes W$, which we take as $\mathcal K$.

If $l\leq n-1$, then the first line of \cref{eq:adjoint_of_pure_isometry}, extended to all integer values of the labelling indices, describes $\Ue_i^\ast$ for $i=1,\dotsc,l$, whereas \cref{eq:adjoint_of_unitary_from_definition}, also extended to all integer values of the labelling indices, describes $\Ue_i^\ast$ for $i=l+1,\dotsc,n$. If $l=n$, then the first line of \cref{eq:adjoint_of_pure_isometry}, extended to all integer values of the labelling indices and where now $l=n$, describes $\Ue_i^\ast$ for $i=1,\dotsc,n$.

If $l\leq n-1$, the computational part of the proof of \cref{res:isometries_satisfy_relations} shows that $\Ue_1,\dotsc, \Ue_l$ satisfy the pertinent relations among themselves, and the computational part of the proof of \cref{res:isometries_and_unitaries_satisfy_relations} shows that this is likewise true for $\Ue_1,\dotsc, \Ue_l$ on the one hand, and $\Ue_{l+1},\dotsc, \Ue_n$ on the other hand. As with \cref{res:unitaries_satisfy_relations}, it is clear that $\Ue_{l+1},\dotsc, \Ue_n$ satisfy the pertinent relations among themselves. This establishes part~\partref{part:dilated_tuple_is_doubly_non-commutative} if $l\leq n-1$.

If $l=n-1$, the computational part of the proof of \cref{res:isometries_satisfy_relations} shows that $\Ue_1,\dotsc, \Ue_n$ satisfy the pertinent relations. This establishes part~\partref{part:dilated_tuple_is_doubly_non-commutative} if $l=n$.

The remaining statements follow by inspection. Economising on this a little, we add that, if $l\leq n-1$, it is a direct consequence of the unitarity of $\Ue_l,\Ue_{l+1},\dotsc,\Ue_n$ (which implies that they reduce $H$) that the restriction of $\Ue_{l+1}^\ast,\dotsc,\Ue_n^\ast$ to $H$ coincides with $U_{l+1}^\ast,\dotsc,U_{n}^\ast$, respectively.

\end{proof}

An appeal to \cref{res:space_decomposition_types_of_actions}, combined with the obvious generalisation of \cref{res:dilation_for_fixed_A} to arbitrary $A\subseteq\{1,\dotsc,n\}$ and with a Hilbert direct sum argument, then yields the following.

\begin{theorem}[Dilation theorem]\label{res:dilation}
Let $\Vtuple$ be an $n$-tuple of \dnci\ on a Hilbert space $H$. Then there exists a Hilbert space $\mathcal K$ containing $H$, with projection $P_H:\mathcal K\to H$, and unitary operators $\Ue_1,\dotsc,\Ue_n$ on $\mathcal K$ such that
\begin{enumerate}
\item $(\Ue_1,\dotsc,\Ue_n)$ is an $n$-tuple of \dnci;
\item $\Ue_1,\dotsc,\Ue_n$ leave $H$ invariant;
\item The restriction of $\Ue_1,\dotsc, \Ue_n$ to $H$ is $\V_1,\dotsc,\V_{n}$, respectively;
\item $V_i^\ast=(P_H \circ \Ue_i^\ast)|_H$ for $i=1,\dotsc,n$.
\end{enumerate}
\end{theorem}

\begin{remark}\label{rem:dilation}
	On taking all structure constants equal to one, \cref{res:dilation} specialises to a dilation theorem for finitely many doubly commuting isometries: these can be extended to commuting unitary operators on an enveloping space. They can, in fact, even be extended to doubly commuting unitary operators; the pertinent extra relations, however, are already automatic by Fuglede's theorem. It should be noted here that, in the commutative domain, a much stronger version is known to be true than this specialisation of \cref{res:dilation}. Any (not necessarily finite) family of commuting (not necessarily doubly commuting) isometries on a Hilbert spaces can be extended to a family of commuting unitary operators on an enveloping Hilbert space. This result goes back to It\^o  \cite[Theorem~3]{ito:1958} and Brehmer \cite{brehmer:1961}; see also \cite[Theorem~I.6.2]{sz_nagy_foias_bercovici_kerchy_HARMONIC_ANALYSIS_OF_OPERATORS_ON_HILBERT_SPACE_SECOND_EDITION:2010}. As is well known, this fact implies the validity of a von Neumann inequality for polynomials in several commuting isometries on a Hilbert space.
\end{remark}

\subsection*{Acknowledgements} The results in this paper were partly obtained during a research visit of the first author to the University of Lisbon. The kind hospitality of the Instituto Superior T\'{e}cnico is gratefully acknowledged. The second author was partially funded by FCT/Portugal through the Centre for Mathematical Analysis, Geometry, and Dynamical Systems (CAMGSD). We thank the referee for the careful reading of the manuscript and the constructive suggestions.



\bibliographystyle{amsplain}

\bibliography{general_bibliography}


\end{document}